%% file: RamseyTransference.tex
\documentclass[11pt]{article}

\usepackage{amsmath, amssymb, amsthm}
\usepackage{tikz}

\theoremstyle{plain}
\newtheorem{thm}{Theorem}[section]
\newtheorem*{thm*}{Theorem}

\newtheorem{lem}[thm]{Lemma}
\newtheorem*{lem*}{Lemma}
\newtheorem{dfn}[thm]{Definition}
\newtheorem{cor}[thm]{Corollary}

\newtheorem{ques}[thm]{Question}

\newcommand{\codeg}{\text{codeg}}
\newcommand{\BBE}{\mathbb{E}}
\newcommand{\BFP}{\mathbf{P}}

\date{}
\title{\vspace{-0.7cm}A transference principle for Ramsey numbers of bounded degree graphs}

\author{
Choongbum Lee \thanks{Department of Mathematics,
MIT, Cambridge, MA 02139-4307. Email: cb\_lee@math.mit.edu.
Research supported by NSF Grant DMS-1362326.}
}

\begin{document}

\maketitle

\begin{abstract}
We investigate Ramsey numbers of bounded degree graphs and
provide an interpolation between known results on the Ramsey numbers of 
general bounded degree graphs 
and bounded degree graphs of small bandwidth. Our main theorem implies that
there exists a constant $c$ such that for every $\Delta$, 
there exists $\beta$ such that if $G$ is a graph with maximum degree
at most $\Delta$ having a homomorphism $f$ into a graph $H$ of maximum degree at most $d$ where
$|f^{-1}(v)| \le \beta n$ for all $v \in V(H)$, then
the Ramsey number of $G$ is at most $c^{d \log d} n$. 
A construction of Graham, R\"odl, and Ruci\'nski shows that the statement above
holds only if $\beta \le (c')^{\Delta}$ for some constant $c' < 1$.
We further study the parameter $\beta$ using
a density-type embedding theorem for bipartite graphs of small bandwidth. 
This theorem may be of independent interest.
\end{abstract}

\section{Introduction} \label{sec:intro}

The {\em Ramsey number} of a graph $G$, denoted $r(G)$, is the minimum number $n$ such that
every edge-coloring of $K_n$ using two colors admits a monochromatic
copy of $G$. 
It was first studied in the seminal paper of Ramsey \cite{Ramsey} which established
that the Ramsey number of the complete graph $K_k$ on $k$ vertices is finite for all positive
integers $k$.
Since then, Ramsey theory, the study of various results that can be grouped under the common theme 
``every large system has a well-organized subsystem'', flourished and
became one of the most active fields of research in combinatorics. It is a beautiful
field with many questions still remaining to be answered and has deep connections 
to other fields such as logic, geometry, and computer science. See the classical book of
Graham, Rothschild, and Spencer \cite{GrRoSp} for a comprehensive overview of the field, or a survey
of Conlon, Fox, and Sudakov \cite{CoFoSu15} for recent developments in graph Ramsey theory.

In this paper we study the Ramsey number of bounded degree graphs.
The history of such study can be traced back to a paper of Burr and Erd\H{o}s \cite{BuEr} from 1975
which predicted that the behavior of Ramsey numbers of sparse graphs will be dramatically different 
from that of the complete graph (the Ramsey number of complete graphs is exponential in terms of the number of vertices \cite{ErSz, Erdos}).
A graph $G$ is {\em $d$-degenerate} if all its subgraphs has a vertex of degree at most $d$.
In their paper, Burr and Erd\H{o}s conjectured that
for all $d$, there exists a constant $c = c(d)$ such that $r(G) \le c(d) n$ for
all $n$-vertex $d$-degenerate graphs $G$. This conjecture is still open 
(see \cite{FoSu, Lee}).

In 1983, Chv\'atal, R\"odl, Szemer\'edi, and Trotter \cite{ChRoSzTr} showed that the
Burr-Erd\H{o}s conjecture holds if the degeneracy condition is replaced with a bounded degree condition.
More precisely, they showed that for all $\Delta$, there exists a constant $c = c(\Delta)$ such that
$r(G) \le c(\Delta) n$ for all $n$-vertex graphs $G$ of maximum degree at most $\Delta$.
Their proof relied on the regularity lemma and gave a tower-type dependency between $c(\Delta)$ and $\Delta$. 
Since then, the bound on $c(\Delta)$ has been improved by Eaton \cite{Eaton}, 
Graham, R\"odl, and Ruci\'nski \cite{GrRoRu00, GrRoRu01}, and by 
Conlon, Fox, and Sudakov \cite{CoFoSu} who showed that there exists a constant $c$
such that $r(G) \le c^{\Delta \log \Delta} n$
holds for $n$-vertex graphs $G$ of maximum degree at most $\Delta$.
For bipartite graphs, independently, 
Conlon \cite{Conlon}, and  Fox and Sudakov \cite{FoSu0} 
showed that a better bound
$r(G) \le c^{\Delta}n$ holds (for some other constant $c$).
On the other hand, Graham, R\"odl, and Ruci\'nski \cite{GrRoRu01} showed that
there exists a constant $c > 1$ such that for all large enough $n$, there exists a bipartite graph 
$G$ with maximum degree at most $\Delta$ satisfying $r(G) \ge c^{\Delta}n$.

In some cases the constant is known to be significantly smaller.
The {\em bandwidth} of an $n$-vertex graph $G$ is the minimum integer $b$ for which there exists a labelling 
of the vertices by $[n]$ such that $|i-j| \le b$ holds for all edges $\{i,j\} \in E(G)$.
Allen, Brightwell, and Skokan \cite{AlBrSk} showed that if $G$ is 
an $n$-vertex with maximum degree at most $\Delta$ and bandwidth at most $\beta n$, then
$r(G) \le (2\chi(G) + 4)n \le (2\Delta + 6)n$.
Despite the rather wide gap between the two constants $c^{\Delta}$ and $2\Delta+6$,
not much is known about the constant that dictates the Ramsey number of 
graphs of bounded maximum degree.
Since it is known that a graph has small bandwidth if and only if it has poor expansion property
(see \cite{BoPrTaWu} for more detail) and
the example of Graham, R\"odl, and Ruci\'nski is a good expander, one can reasonably guess that the
constant $c_G$ for which $r(G) \le c_G \cdot n$ depends on the expansion property of 
$G$. In this paper, we study the relation between the constant $c_G$ and the structure
of the graph $G$ in further depth.

\medskip

Throughout the paper, when considering 2-edge-coloring of a graph,
we tacitly assume that the two colors are red/blue, respectively,
and refer to the subgraph consisting of the red edges as the {\em red graph}, and
of the blue edges as the {\em blue graph}.
A {\em (vertex) weighted graph} is a pair $(G, w)$ of a graph $G$
and a weight function $w : V(G) \rightarrow [0,1]$. For a set $X \subseteq V(G)$, 
define $w(X) = \sum_{x \in X} w(x)$. For two graphs $G$ and $H$, a {\em homomorphism}
from $G$ to $H$ is a map $f : V(G) \rightarrow V(H)$ such that $\{f(v), f(w)\} \in E(H)$ 
whenever $\{v,w\} \in E(G)$.

\begin{dfn}
The {\em Ramsey number of a weighted graph $(G,w)$}, denoted $\hat{r}(G,w)$ is the minimum integer $n$ satisfying
the following: for every 2-edge-coloring of $K_n$, there exists a homomorphism $f$ 
from $G$ to the red graph, or to the blue graph, for which $w(f^{-1}(v)) \le 1$ holds for
all $v \in V(K_n)$. We simply denote $\hat{r}(G, w)$ as $\hat{r}(G)$ when the weight
function is clear from the context.
\end{dfn}

Note that the Ramsey number of weighted graphs generalizes the 
Ramsey number of graphs since given a graph we can always
consider a constant weight function assigning weight one to all vertices.
For this weight function, the Ramsey number in the traditional sense is equal
to the Ramsey number as defined above. 
This observation in particular implies that $\hat{r}(G,w)$ is 
finite for all weighted graphs $(G,w)$.
Further note that if $G$ is the complete graph, then for all weight functions $w$,
the Ramsey number of $(G,w)$ equals the Ramsey number of $G$ since a homomorphism 
from a complete graph to a graph with no loops is necessarily injective.
On the other hand suppose that $G$ is $k$-colorable and suppose that $w$ is a weight function
where $w(X) \le 1$ for each of the $k$ color classes $X$ of $G$.
Then one can easily check that $\hat{r}(G) \le r(K_k) \le {2k \choose k}$.
Therefore both the structure of $G$ and the weight function $w$ plays an important role
in determining the Ramsey number of weighted graphs.
However we will later see that for bounded degree graphs $G$, 
the Ramsey number of $(G,w)$ is mostly determined by the total weight 
$w(V(G))$ of the graph.

We consider another generalization of Ramsey numbers, 
implicitly studied in \cite{AlBrSk}, where the host graph is a graph
of large minimum degree instead of the complete graph.

\begin{dfn}
For a positive real $\varepsilon$ and a weighted graph $(G,w)$, define the
{\em $\varepsilon$-stable Ramsey number} $\hat{r}_\varepsilon(G,w)$ as the minimum 
integer $n$ satisfying the following:
for every graph $\Gamma$ on $n$ vertices of minimum degree at least $(1-\varepsilon)n$, 
for every 2-edge-coloring of $\Gamma$, there exists a homomorphism $f$ 
from $G$ to the red graph, or to the blue graph, for which $w(f^{-1}(v)) \le 1$ holds for
all $v \in V(\Gamma)$.
\end{dfn}

The stable Ramsey number generalizes Ramsey number since 
$\hat{r}(G,w) = \hat{r}_{\varepsilon}(G,w)$ holds
for every weighted graph $(G,w)$ if $\varepsilon < \frac{1}{\hat{r}(G,w) - 1}$.
However, given a weighted graph $(G,w)$, the $\varepsilon$-stable Ramsey number
does not necessarily exist. For example if $G$ is an $r$-partite graph, then 
$\hat{r}_{\varepsilon}(G,w)$ does not exist for $\varepsilon \ge \frac{1}{r-1}$ since we can take
the host graph $\Gamma$ to be a complete $(r-1)$-partite graph. In fact 
$\hat{r}_{\varepsilon}(G,w)$ is finite if and only if $\varepsilon < \frac{1}{r(K_{\chi(G)})-1}$
(see Section~\ref{sec:remarks}).
The following theorem extends a theorem of Conlon, Fox, and Sudakov
and shows that for bounded degree graphs, 
the stable Ramsey number is mostly determined by the total weight of the graph.

\begin{thm} \label{thm:weight_cfs}
There exist constants $c$ such that the following holds for every natural number $\Delta$ 
and positive real number $\varepsilon$ satisfying $\varepsilon < c^{-\Delta \log \Delta}$.
If $(G,w)$ is a weighted graph with maximum degree at most $\Delta$, then
$\hat{r}_\varepsilon(G) \le c^{\Delta \log \Delta} \cdot w(V(G))$. 
\end{thm}

The following result is the main theorem of this paper studying the Ramsey number of bounded
degree graphs.
It roughly asserts that if $G$ is a subgraph of a blow-up of $H$, then the Ramsey number of 
$G$ can be described in terms of the Ramsey number of $H$. 

\begin{thm} \label{thm:transference}
For all $\Delta, \xi$ and $\varepsilon$, 
there exists $\beta$ and $n_0$ such that the following holds for all $n \ge n_0$.
Let $G$ and $H$ be graphs where $G$ has $n$ vertices and maximum degree at most $\Delta$.
Suppose that there exists a homomorphism $f$ from $G$ to $H$ for which $|f^{-1}(v)| \le \beta n$
for all $v \in V(H)$.
Then for the weight-function of $H$ defined by $w(v) = \frac{1}{\beta n}|f^{-1}(v)|$ we have
$r(G) \le (1+\xi)\hat{r}_\varepsilon(H,w) \cdot \beta n$.
\end{thm}

Note that the weight function $w$ defined in Theorem~\ref{thm:transference}
satisfies $w(V(H)) = \frac{1}{\beta}$.
A {\em wheel graph} $W_{k}$ is a graph with $k$ vertices consisting of a 
cycle on $k-1$ vertices and a vertex adjacent to all vertices on the cycle. 
In Section~\ref{sec:weighted_ramsey}, we will see that there exist constants $\varepsilon$
and $c$ such that $\hat{r}_\varepsilon(W_k, w) \le c \cdot w(V(W_k))$ for all $k$ and all weight 
functions $w : V(W_k) \rightarrow [0,1]$.
Hence if $G$ has maximum degree at most $\Delta$
and a homomorphism $f$ into a wheel graph $W_k$
where $|f^{-1}(v)| \le \beta n$ for all $v \in V(W_k)$, then 
by Theorem~\ref{thm:transference} above with $\xi = 1$, we obtain 
$r(G) \le 2\hat{r}_\varepsilon(W_k,w) \cdot \beta n \le 2c n$
(see Figure~\ref{fig:fig_wheel}). Note in particular that the constant
does not depend on $\Delta$. This is in sharp contrast with the bound
$r(G) \le c^{\Delta \log \Delta} n$ (the constant $c$ is different from above)
that we obtain through the theorem of Conlon, Fox, and Sudakov. 
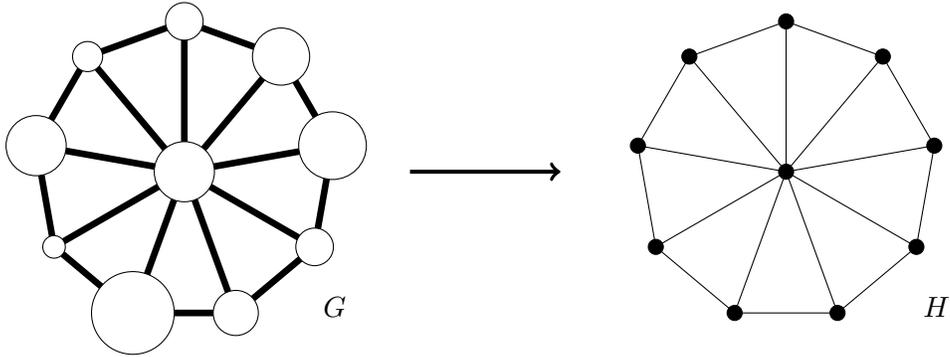
\begin{figure}
  \centering
  \begin{tabular}{ccc}
    \input{fig-hom}
  \end{tabular}
  \caption{A graph of maximum degree at most $\Delta$ and a homomorphism into a wheel graph.}
 \label{fig:fig_wheel}
\end{figure}
As another example, if $H$ has maximum degree at most $d$, then by Theorem~\ref{thm:weight_cfs} 
we see that $\hat{r}_\varepsilon(H,w) \le c^{d \log d}\frac{1}{\beta}$ for small enough $\varepsilon$.
By applying Theorem~\ref{thm:transference} with the $\xi = 1$ and $\varepsilon < c^{d\log d}$, 
we obtain the following corollary.

\begin{cor} \label{cor:transference}
There exists a constant $c$ such that for all $\Delta$, 
there exists $\beta$ and $n_0$ such that the following holds for all $n \ge n_0$. 
Let $G$ be a $n$-vertex graph with maximum degree at most $\Delta$, and $H$ be a graph with
maximum degree at most $d$. Suppose that there exists
a homomorphism $f$ from $G$ to $H$ for which $|f^{-1}(v)| \le \beta n$ for all $v \in V(H)$.
Then $r(G) \le c^{d \log d} n$.
\end{cor}

It is known (implicitly in \cite{BoScTa}) that for all $r$, there exists
a constant $c > 1$ such that if $G$ is an $r$-partite graph of 
bandwidth at most $\beta n$, 
then there exists a homomorphism $f$ from $G$ to the $r$-th power of a path of length 
$\frac{1}{c \beta}$ where $|f^{-1}(v)| \le c \beta n$ for all $v$.
Allen, Brightwell, and Skokan's result $r(G) \le (2\chi(G)+4)n$ mentioned above then follows 
from Theorem~\ref{thm:transference} and a
bound on the Ramsey number of power of paths  (in fact the proof of Theorem~\ref{thm:transference} is based on a generalization of their proof).

The necessity of forcing $|f^{-1}(v)| \le \beta n$ for some small constant $\beta$
can be seen from the example of Graham, R\"odl, and Ruci\'nski. 
They proved that there exists a constant $c < 1$ such that
for all $\Delta$ and large enough $n$, there exists 
a $c^{\Delta}n$-vertex bipartite graph $G$ of maximum degree at most $\Delta$ for which
$r(G) > n$.
If $\beta \ge 20c^{\Delta}$ in Theorem~\ref{thm:transference}, then
there exists a homomorphism $f$ from $G$ to $K_2$ such that $|f^{-1}(v)| \le \frac{1}{20}\beta n$
for both vertices $v$ of $K_2$. Since $\hat{r}_\varepsilon(K_2,w) = 2$, 
Theorem~\ref{thm:transference} (if true) will imply that $r(G) \le (1+\xi)2\beta n < n$ which is a contradiction.
Thus we see that $\beta$ must be at most $20c^{\Delta}n$ in Theorem~\ref{thm:transference}.
On the other hand, the bound on $\beta$ that we obtain in Theorem~\ref{thm:transference} has a tower-type dependency on $\Delta$.
It would be interesting to determine the best possible value of $\beta$ that we can take.
The following density embedding theorem has an interesting implication towards this problem.

\begin{thm} \label{thm:densityembedding}
Let $G$ be an $n$-vertex graph of minimum degree at least $(\delta + \alpha)n$. Then $G$ contains
all bipartite graphs $H$ on at most $\delta n$ vertices with maximum degree at most $\Delta$ and bandwidth at most $\frac{1}{256\Delta}\alpha^{6\Delta+1}n$.
\end{thm}

Theorem~\ref{thm:densityembedding} can be seen as an extension of density embedding theorems of bipartite graphs proved by Conlon \cite{Conlon}, and Fox and Sudakov \cite{FoSu0}, and may be of independent interest.
Note that Theorem~\ref{thm:densityembedding} is asymptotically tight in terms of
the number of vertices of $H$.
As we will later see (Corollary~\ref{cor:bandwidth}), it implies that 
$r(G) \le (4+\varepsilon)n$ if
$G$ is a bipartite graph with maximum degree at most $\Delta$ and bandwidth at most $c^{\Delta}n$
for some positive constant $c < 1$.
Similar result can be obtained by the theorem of Allen, Brightwell, and Skokan
but with a worse bound on the bandwidth. This corollary implies that a transference-type
result holds even when $\beta$ is as large as $c^\Delta n$ for the special case when
$G$ is a bipartite graph with small bandwidth.

A $d$-dimensional hypercube $Q_d$ is a graph with vertex set $\{0,1\}^d$ where
two vertices are adjacent if and only if they differ in exactly one coordinate.
Since the bandwidth of $Q_d$ is known to be $O(\frac{2^d}{d})$, 
Theorem~\ref{thm:densityembedding} is closely related to another conjecture of
Burr and Erd\H{o}s stating that there exists a constant $c$ such that $r(Q_d) \le c 2^d$ holds for all natural numbers $d$.
Unfortunately the bandwidth condition in Theorem~\ref{cor:bandwidth}
is too weak to imply the conjecture.

\medskip

The rest of the paper is organized as follows. 
In Section~\ref{sec:transference} we prove Theorem~\ref{thm:transference} using a variant of the blow-up lemma
whose proof we defer to a later section.
In Section~\ref{sec:weighted_ramsey} we establish a bound on the weighted Ramsey number
of wheel graph and prove Theorem \ref{thm:weight_cfs}. 
In Section~\ref{sec:blowup} we prove the variant of the blow-up lemma used in Section~\ref{sec:transference}. 
In Section~\ref{sec:bipartite_ramsey}, we prove Theorem~\ref{thm:densityembedding} and then conclude with some remarks in Section~\ref{sec:remarks}.

\medskip

\noindent \textbf{Notation}.
A graph $G = (V,E)$ is given by a pair of vertex set $V$ and edge set $E$. 
For a vertex $v \in V$, define $\deg(v)$ as the degree of $v$, 
and for a set $X \subseteq V$, define $\codeg(X)$ as the number of common neighbors
of the vertices in $X$. For two vertices $v,w \in V$, we define $\codeg(v,w) = \codeg(\{v,w\})$.
For a set $X \subset V$, define $G[X]$ as the subgraph of $G$ induced on $X$.
For a pair of sets $X, Y \subseteq V$, define $e(X,Y)$ as the number of pairs
$(x,y) \in X \times Y$ that form an edge in $G$. When $X,Y$ are disjoint sets,
define $d(X,Y) = \frac{e(X,Y)}{|X||Y|}$.
When there are several graphs
under consideration, we often use subscript such as in $e_G(X,Y)$ to clarify
the graph that we are referring to.
For two graphs $H$ and $G$, an {\em embedding} of $H$ to $G$ is an injective
map $f : V(H) \rightarrow V(G)$ for which $\{f(v), f(w)\} \in E(G)$ 
whenever $\{v,w\} \in E(H)$. 
An {\em embedding} of a weighted graph $(H,w)$ into a graph $G$ is a
map $f : V(H) \rightarrow V(G)$ for which $\{f(v), f(w)\} \in E(G)$
whenever $\{v,w\} \in E(H)$ and $w(f^{-1}(v)) \le 1$ for all $v \in V(H)$.
For a set $X \subseteq V(H)$, a 
{\em partial embedding on $X$} is an embedding of $H[X]$ to $G$.
For a finite set $X$ and a natural number $n$, we use the notation 
$X^n$ to denote the product space $X \times X \times \cdots \times X$
where product is taken $n$ times. Equivalently, $X^n$ is the set of
ordered $n$-tuples of elements of $X$.

We use $\log$ without subscript to denote base 2 logarithm.
We omit floor and ceiling whenever they are not crucial.
Throughout the paper, we use constants with subscripts such as in
$\beta_{2.3}$ to indicate that $\beta$ is the constant coming from
Theorem/Corollary/Lemma/Proposition 2.3.

\section{Transference principle}  \label{sec:transference}

Let $G$ be a graph on $n$ vertices.
A pair of disjoint vertex subsets $(X,Y)$ is {\em $\varepsilon$-regular} if for all $X' \subseteq X$ and $Y' \subseteq Y$
satisfying $|X'| \ge \varepsilon|X|$ and $|Y'| \ge \varepsilon|Y|$, we have
$\left| d(X,Y) - d(X',Y') \right| \le \varepsilon$. 
A partition $V(G) = V_0 \cup V_1 \cup \cdots \cup V_k$ is {\em $\varepsilon$-regular} if
(i) $|V_0| \le \varepsilon n$, (ii) $|V_i| = |V_j|$ for all $i,j \ge 1$, and (iii)
for each $i \in [k]$, there exists at most $\varepsilon k$ indices $j \in [k]$ for which
$(V_i, V_j)$ is not $\varepsilon$-regular. 
\footnote{We deviate from the standard practice enforcing at most $\varepsilon k^2$ 
pairs that are not $\varepsilon$-regular.}
We define the {\it $\varepsilon$-reduced graph} of a
partition $\{V_i\}_{i =0}^{k}$ as the graph 
with vertex set $[k]$ where $V_i$ and $V_j$ forms an edge if and only if
the pair $(V_i, V_j)$ is $\varepsilon$-regular. Note that condition (iii) is equivalent to 
saying that the $\varepsilon$-reduced graph of the partition has minimum degree at least $(1-\varepsilon)k$.
For a real number $\delta$, we define the {\it $(\varepsilon,\delta)$-reduced graph} of a 
partition $\{V_i\}_{i=0}^{k}$ as the graph with vertex set $[k]$ where $V_i$ and $V_j$ forms 
an edge if and only if the pair $(V_i, V_j)$ is $\varepsilon$-regular with density at least
$\delta$.
The celebrated regularity lemma asserts that all large graphs admit an $\varepsilon$-regular partition
(see \cite{KoSi} for the version of the regularity lemma as stated here).

\begin{thm} \label{thm:regularity}
For all $\varepsilon$ and $t$, there exists $n_0 = n_0(\varepsilon,t)$ and $T = T(\varepsilon,t)$
such that the following holds for all $n \ge n_0$. 
Every $n$-vertex graph $G$ admits an $\varepsilon$-regular partition into $k$ parts
where $t \le k \le T$.
\end{thm}

We will later need an $\varepsilon$-regular partition with a prescribed number of parts. 
Such partition can be produced by taking a random refinement of an $\varepsilon$-regular
partition obtained through the regularity lemma.
The following lemma, proved in \cite{GeKoRoSt}, can be used to verify that such partition indeed works.
It asserts that a typical pair of subsets of a regular pair is regular.

\begin{lem} \label{lem:inherit_regular}
For $0 < \beta, \varepsilon < 1$, there exists $\varepsilon_0 = \varepsilon_0(\beta, \varepsilon)$
and $C = C(\varepsilon)$ such that for all $\varepsilon' \le \varepsilon_0$ and $\delta$,  
every $\varepsilon'$-regular pair $(X,Y)$ of density at least $\delta$ 
satisfies that, for every $q \ge C \delta^{-1}$, 
the number of sets $Q \subseteq X$ of cardinality $q$ that form an $\varepsilon$-regular
pair of density at least $\delta$ with $Y$ is at least $(1-\beta^{q}){|V_1| \choose q}$.
\end{lem}

By combining Theorem~\ref{thm:regularity} and Lemma~\ref{lem:inherit_regular}, we can prove
a regularity lemma which outputs a partition with a prescribed number of parts.

\begin{lem} \label{lem:regularity_fixednumber}
For all $\varepsilon$, there exists $T = T(\varepsilon)$ such that for all 
$k \ge T$ there exists $n_0(\varepsilon, k)$ such that the following holds for all $n \ge n_0$.
Every $n$-vertex graph $G$ admits an $\varepsilon$-regular partition 
$V_0 \cup V_1 \cup \cdots \cup V_k$.
\end{lem}
\begin{proof}
Let $\varepsilon_0 = \min\{(\varepsilon_0)_{\ref{lem:inherit_regular}}(\frac{1}{2}, \varepsilon), \frac{\varepsilon}{2}\}$,
$C = C_{\ref{lem:inherit_regular}}(\varepsilon)$, and
$T = \frac{2}{\varepsilon} T_{\ref{thm:regularity}}(\varepsilon_0,\frac{1}{\varepsilon_0})$.
Suppose that an integer $k \ge T$ is given.
Apply Theorem~\ref{thm:regularity} with $\varepsilon_{\ref{thm:regularity}} = \varepsilon_0$
and $t_{\ref{thm:regularity}} = \frac{1}{\varepsilon_0}$
to find an $\varepsilon_0$-regular partition
$V_0 \cup V_1 \cup \cdots \cup V_r$ where $\frac{1}{\varepsilon_0} \le r \le \frac{\varepsilon}{2}T$.
Define $s = \left\lceil \frac{k}{r} \right\rceil$ and note that $s \ge \frac{T}{r} \ge \frac{2}{\varepsilon}$.

For each $i \in [r]$, we may assume that $|V_i|$ is divisible by $s$ 
by moving at most $s-1$ vertices from $V_i$ to $V_0$ if necessary.
For each $i \in [r]$, let $V_i = V_{i,1} \cup V_{i,2} \cup \cdots \cup V_{i,s}$
be a partition chosen uniformly at random where $|V_{i,j}| = \frac{1}{s}|V_i|$ for all $j \in [s]$.
By Lemma~\ref{lem:inherit_regular}, if $(V_i, V_{i'})$ is $\varepsilon_0$-regular,
then for all $j,j' \in [s]$, the probability that $(V_{i,j}, V_{i',j'})$ forms 
an $\varepsilon$-regular pair
is at least $1 - 2^{-\Omega(n)}$. Hence by the union bound, we can find partitions
$V_i = V_{i,1} \cup V_{i,2} \cup \cdots \cup V_{i,s}$ for each $i \in [r]$ so 
that $(V_{i,j}, V_{i',j'})$ forms an $\varepsilon$-regular pair 
whenever $(V_{i}, V_{i'})$ forms an $\varepsilon_0$-regular pair.
Thus each $V_{i,j}$ forms an $\varepsilon$-regular pair with at least
$(1-\varepsilon_0)rs$ other sets $V_{i',j'}$.
Arbitrarily remove $rs - k \le r-1$ parts $(i,j) \in [r] \times [s]$, combine the removed sets with $V_0$
and re-label the sets so that we obtain a partition
$U_0 \cup U_1 \cup \cdots \cup U_k$ of the vertex set. 
For each $i \in [k]$, there are at most $\varepsilon_0 rs \le \frac{\varepsilon}{2} rs \le \varepsilon k$
other indices $i' \in [k]$ for which $(U_i, U_{i'})$ is not $\varepsilon$-regular. 
Moreover, $|U_0| \le \varepsilon_0 n + (s-1)r + (r - 1)\frac{n}{sr} \le \varepsilon n$ and therefore
we found a partition with the desired properties.
\end{proof}

The blow-up lemma, developed by Koml\'os, S\'ark\"ozy, and Szemer\'edi \cite{KoSaSz},
is a powerful tool used in embedding large subgraphs. Informally, quoting
Koml\'os, S\'ark\"ozy, and Szemer\'edi, it asserts that, ``regular pairs behave
like complete bipartite graphs from the point of view of bounded degree subgraphs.''.
We use the following version of the blow-up lemma.

\begin{lem} \label{lem:blowup}
For all $\xi, \delta, \Delta$, there exists $\varepsilon = \varepsilon(\xi, \delta, \Delta)$ 
such that the following holds for all natural numbers $k$ 
if $m \ge m_0$ for some sufficiently large $m_0 = m_0(k,\varepsilon)$.
Let $G$ be a graph with maximum degree at most $\Delta$.
Let $\Gamma$ be a graph with a vertex partition $\{V_i\}_{i=1}^{k}$ satisfying $|V_i| \ge (1+\xi)m$ for all $i \in [k]$, and let $R$ be its $(\varepsilon, \delta)$-reduced graph.
Suppose that there exists a homomorphism $f$ from $G$ to $R$ where $|f^{-1}(i)| \le m$ for all $i \in [k]$.
Then there exists an embedding of $G$ to $\Gamma$.
\end{lem}

Lemma~\ref{lem:blowup} differs from the original 
version of the blow-up lemma in that the restriction on $\varepsilon$ does not 
depend on $R$. It is a subtle but crucial difference.
The proof of Lemma~\ref{lem:blowup} is rather technical 
and hence to avoid unnecessary distraction, we provide it in Section~\ref{sec:blowup}.

Theorem~\ref{thm:transference} follows from Lemmas~\ref{lem:regularity_fixednumber}
and \ref{lem:blowup}. We re-state the theorem here.

\begin{thm*} 
For all $\Delta, \xi$ and $\varepsilon$, there exists $\beta$ and $n_0$ such that the following holds for all $n \ge n_0$.
Let $G$ and $H$ be graphs where $G$ has $n$ vertices and maximum degree at most $\Delta$.
Suppose that there exists a homomorphism $f$ from $G$ to $H$ for which $|f^{-1}(v)| \le \beta n$
for all $v \in V(H)$.
Then for the weight-function of $H$ defined by $w(v) = \frac{1}{\beta n}|f^{-1}(v)|$ we have
$r(G) \le (1+\xi)\hat{r}_\varepsilon(H,w) \cdot \beta n$.
\end{thm*}
\begin{proof}
By reducing $\varepsilon$ if necessary, 
we may assume that 
$\varepsilon \le \min\left\{\varepsilon_{\ref{lem:blowup}}(\frac{\xi}{2}, \frac{1}{2}, \Delta), \frac{\xi}{2(1+\xi)}\right\}$.
By the theorem of Conlon, Fox, and Sudakov mentioned in the introduction,
there exists a constant $c$ for which $r(G) \le c^{\Delta \log \Delta}n$. 
Define $\beta = T_{\ref{lem:regularity_fixednumber}}(\varepsilon)^{-1}$ and $k = \hat{r}_\varepsilon(H)$.
Define $n_0 = \max \{(n_0)_{\ref{lem:regularity_fixednumber}}(\varepsilon, \beta^{-1}c^{\Delta \log \Delta}),
(m_0)_{\ref{lem:blowup}}(k, \varepsilon) \}$.
By definition we have 
$k = \hat{r}_{\varepsilon}(H) \ge w(V(H)) = \frac{1}{\beta n} |V(G)| = \beta^{-1}$.
Furthermore since $r(G) \le c^{\Delta \log \Delta}n$, the conclusion holds if 
$\hat{r}_\varepsilon(H) \ge \beta^{-1}c^{\Delta \log \Delta}$. 
Thus we may assume that $k = \hat{r}_\varepsilon(H) < \beta^{-1}c^{\Delta \log \Delta}$. 
Thus $\beta^{-1} \le k < \beta^{-1} c^{\Delta \log \Delta}$.

Let $N = (1+\xi)\hat{r}_\varepsilon(H) \cdot \beta n$.
Suppose that we are given a red/blue coloring of $K_N$.
Let $\Gamma_r$ and $\Gamma_b$ be the red graph and blue graph, respectively.
By Theorem \ref{thm:regularity}, there exists an $\varepsilon$-regular partition 
$V_0 \cup V_1 \cup \cdots \cup V_k$ of $\Gamma_r$ (note that it also is an
$\varepsilon$-regular partition of $\Gamma_b$).
Consider the $\varepsilon$-reduced graph $R$ of the partition, and color the edges with
red and blue so that an edge $\{i,j\}$ is red if the red edge density of the
pair $(V_i, V_j)$ is at least $\frac{1}{2}$ and blue otherwise.
Since $R$ has minimum degree at least $(1-\varepsilon)k$, by the definition 
of $\hat{r}_{\varepsilon}(H)$, there exists a homomorphism $g$ from $H$
to the red subgraph of $R$ (or the blue subgraph of $R$) such that $|g^{-1}(i)| \le 1$, 
and thus $|f^{-1}(g^{-1}(i))| \le \beta n$
for all $i \in [k]$. Without loss of generality, assume that it is to the red subgraph of $R$.
Note that $h := g \circ f$ is a homomorphism from $G$ to the red subgraph of $R$
satisfying $|h^{-1}(i)| \le \beta n$ for all $i \in [k]$. 
Further note that for each $i \in [k]$, 
we have $|V_i| \ge \frac{1-\varepsilon}{k}N \ge (1+\frac{\xi}{2})n$.
Therefore by Lemma~\ref{lem:blowup}, we can find a copy of $G$ in $\Gamma_r$.
\end{proof}

\section{Weighted Ramsey number}
\label{sec:weighted_ramsey}

\subsection{Wheel graph}

Before proving Theorem~\ref{thm:weight_cfs}, we first show that the weighted Ramsey number of
wheel graphs is small as claimed in the introduction without proof.
Recall that a wheel graph $W_{k}$ is a graph with $k$ vertices consisting of a 
cycle on $k-1$ vertices and a vertex adjacent to all vertices on the cycle. 
Let $w : V(W_k) \rightarrow [0,1]$ be a weight function and define $m = w(V(W_k))$ as the total weight.
Let $C$ be the cycle obtained from $W_k$ by removing the vertex of degree $k-1$.
By restricting the domain of $w$, we may assume that $(C,w)$ is a weighted graph. 
Since a cycle has maximum degree 2, 
Theorem~\ref{thm:weight_cfs} that we will prove in the next subsection
implies that there exist constants $\varepsilon$ and $c$ such that
$\hat{r}_{\varepsilon}(C, w) \le c \cdot w(V(C))$ (where $c$ is a constant
not depending on the order of $C$).
By increasing $c$ and decreasing $\varepsilon$ if necessary, we may assume that 
$c \ge 2$ and $\varepsilon < \frac{c}{8}$.

For simplicity, we assume that $cm$ is an integer. 
For $N = 8c m$, consider a red/blue edge coloring of $\Gamma$, where
$\Gamma$ is a graph with $N$ vertices and minimum degree at least $(1-\frac{\varepsilon}{8})N$. 
Without loss of generality, 
we may assume that there exists a vertex $v_1$ of red degree at least $\lceil \frac{N-1}{2} \rceil \ge 4c m$. 
Let $X$ be an arbitrary set of red neighbors of $v_1$ of size exactly $4c m$
and let $\Gamma_1$ be the graph induced on $X$.
If there exists a vertex $v_2$ of blue degree at least $c m$ in $\Gamma_1$, 
then let $Y$ be an arbitrary set of blue neighbors of $v_2$ in $\Gamma_1$ of size exactly $cm$
and let $\Gamma_2$ be the subgraph of $\Gamma_1$ induced on $Y$.
Note that $\Gamma_2$ has minimum degree at least $cm - \frac{\varepsilon}{8} N = (1-\varepsilon)cm$.
Therefore, we can find a monochromatic copy of $(C, w)$ in $\Gamma_2$. 
If it is red, then together with $v_1$, it forms a monochromatic copy of $(W_k, w)$, and 
if it is blue, then together with $v_2$, it forms a monochromatic copy of $(W_k, w)$.

Hence we may assume that all vertices of $\Gamma_1$ has blue degree at most $cm - 1$ in $\Gamma_1$. 
Since $\Gamma$ has minimum degree at least $(1-\varepsilon)N$, it follows that
$\Gamma_1$ has minimum degree at least $4cm - \varepsilon N$.
Therefore $\Gamma_1$ has minimum red degree at least $3cm - \varepsilon N$. 
Let the vertices of $C$ be $x_1, x_2, \cdots, x_{k-1}$ in decreasing order of weight 
(where ties are broken arbitrarily). 
We will greedily embed the vertices of $C$ according to this order.
Suppose that we finished embedding $x_1, \cdots, x_{i-1}$ and let $\phi$ denote the partial embedding. 
Note that $x_i$ has at most two neighbors in $x_1, \cdots, x_{i-1}$. 
Suppose that it has two neighbors, and let $v, v'$ be the images of these vertices in $\Gamma_1$. 
Let $R$ be the set of common red neighbors of $v$ and $v'$.
By the minimum degree condition of $\Gamma_1$, we know that $|R| \ge 2cm - 2\varepsilon N$.
If there exists a vertex $v'' \in R$ such that $w(\phi^{-1}(v'')) + w(x_i) \le 1$, 
then we define $\phi(x_i) = v''$. Otherwise, we have $w(\phi^{-1}(v'')) > 1 - w(x_i) \ge 0$
for all $v'' \in R$. 
If $w(x_i) \le \frac{1}{2}$, then it implies that $w(\phi^{-1}(v'')) > \frac{1}{2}$
for all $v'' \in R$. On the other hand, if $w(x_i) > \frac{1}{2}$, then 
we have $w(x_j) > \frac{1}{2}$ for all $j < i$. Therefore $w(\phi^{-1}(v'')) > 0$ implies that
$\phi^{-1}(v'') \neq \emptyset$, and thus $w(\phi^{-1}(v'')) > \frac{1}{2}$ for all $v'' \in R$. 
In both cases, we have $w(\phi^{-1}(v'')) > \frac{1}{2}$ for all $v'' \in R$. 
Hence $w(\phi^{-1}(R)) > \frac{1}{2} |R| \ge cm - \varepsilon N > m$ which contradicts the 
fact that $m = w(V(W_k))$. Therefore we can define $\phi(x_i)$ as above and continue the process.
The other case when there are less than two neighbors of $x_i$ in $x_1, \cdots, x_{i-1}$ can
be similarly handled.

\subsection{Bounded degree graphs}

In this section, we adapt the proof of Conlon, Fox, and Sudakov \cite{CoFoSu} to prove Theorem~\ref{thm:weight_cfs}.
We say that a graph $\Gamma$ is {\em bi-$(\varepsilon,\delta)$-dense} if 
for all disjoint pairs of vertex subsets $X, Y \subseteq V(\Gamma)$ of sizes at least
$|X|, |Y| \ge \varepsilon|V(\Gamma)|$, we have $d(X,Y) \ge \delta$.
The following definition is essentially from \cite{CoFoSu} (we added an additional parameter $\delta$).

\begin{dfn}
A graph $\Gamma$ on $N$ vertices is $(\alpha, \beta, \rho, \delta, \Delta)$-dense if there 
is a sequence $U_1, U_2, \cdots, U_s$ of disjoint vertex subsets each of cardinality
at least $\alpha N$ and non-negative integers $d_1, \ldots, d_s$ such that
$d_1 + \cdots + d_s = \Delta - s + 1$, and the following holds:
\begin{itemize}
  \setlength{\itemsep}{1pt} \setlength{\parskip}{0pt}
  \setlength{\parsep}{0pt}
\item[(i)] For all $i \in [s]$, the induced subgraph $\Gamma[U_i]$ is bi-$(\rho^{2d_i}, \delta)$-dense, and
\item[(ii)] for $1 \le i < j \le s$, each vertex in $U_i$ has at least $(1-\beta)|U_j|$ neighbors in $U_j$.
\end{itemize}
\end{dfn}

Note that monotonicity holds in a sense that if a graph is $(\alpha', \beta', \rho', \delta', \Delta')$-dense
and $\alpha' \ge \alpha, \beta' \le \beta, \rho' \le \rho, \delta' \ge \delta, \Delta' \ge \Delta$, then it is also $(\alpha, \beta, \rho, \delta, \Delta)$-dense.
The following lemma was proved in \cite[Lemma 2.2]{CoFoSu}.

\begin{lem} \label{lem:cfs}
Let $D = 2^{h} - 1$ for a non-negative integer $h$, 
and $\rho$ be a fixed real number. 
For all $N \ge 1$, every edge-coloring of $K_N$ with two colors red and blue,
the red graph or the blue graph is $(2^{-2h}\rho^{6D-4h}, 2(D+1)\rho, \rho, \rho, D)$-dense.
\end{lem}

The `stable version' of the lemma above immediately follows for small values of $\varepsilon$.

\begin{lem} \label{lem:cfs_stable}
Let $D = 2^{h} - 1$ for a non-negative integer $h$,
and $\rho, \varepsilon$ be positive real numbers satisfying $\varepsilon < 2^{-2h-1}\rho^{8D-4h+1}$. 
If $\Gamma$ is a graph on $N$ vertices with minimum degree at least $(1-\varepsilon)N$,
then for every edge-coloring of $\Gamma$ with two colors red and blue,
the red graph or the blue graph is $(2^{-2h}\rho^{6D-4h}, 4(D+1)\rho, \rho, \frac{1}{2}\rho, D)$-dense.
\end{lem}
\begin{proof}
Define $\alpha = 2^{-2h}\rho^{6D-4h}$ and $\beta = 2(D+1)\rho$.
Consider an edge-coloring of $K_N$ with three colors, where an edge has color red (blue) if 
it is an edge of color red (blue) in $\Gamma$, and has color green if it is not an edge in $\Gamma$.
Let $\Gamma_r, \Gamma_b, \Gamma_g$ be the graph consisting of red, blue, and green edges respectively.
By Lemma~\ref{lem:cfs}, either $\Gamma_r \cup \Gamma_g$ or $\Gamma_b$ is
$(\alpha, \beta, \rho, \rho, D)$-dense. If $\Gamma_b$ is
$(\alpha, \beta, \rho, \rho, D)$-dense, then the conclusion immediately follows
by monotonicity.

Hence we may assume that $\Gamma' = \Gamma_r \cup \Gamma_g$ is $(\alpha, \beta, \rho, \rho, D)$-dense.
By definition, there exists a sequence $U_1, U_2, \cdots, U_s$ of disjoint vertex subsets each of cardinality
at least $\alpha N$ and non-negative integers $d_1, \ldots, d_s$ such that
$d_1 + \cdots + d_s = D - s + 1$ where the following holds:
\begin{itemize}
  \setlength{\itemsep}{1pt} \setlength{\parskip}{0pt}
  \setlength{\parsep}{0pt}
\item[(i)] For all $i \in [s]$, the induced subgraph $\Gamma'[U_i]$ is bi-$(\rho^{2d_i}, \rho)$-dense, and
\item[(ii)] for $1 \le i < j \le s$, each vertex in $U_i$ has at least $(1-\beta)|U_j|$ neighbors in $U_j$ in $\Gamma'$.
\end{itemize}
We claim that $\Gamma_r[U_i]$ is bi-$(\rho^{2d_i}, \frac{1}{2}\rho)$-dense for all $i \in [s]$, and that
for $1 \le i < j \le s$, each vertex in $U_i$ has at least $(1-2\beta)|U_j|$ neighbors in $U_j$ in $\Gamma_r$.
Note that this proves the lemma. 
To prove the first part of the claim, fix an index $i \in [s]$ and
consider a pair of disjoint vertex subsets $X, Y \subseteq U_i$ of sizes
$|X|, |Y| \ge \rho^{2d_i} |U_i| \ge \rho^{2D}\alpha N = 2^{-2h} \rho^{8D-4h}$. 
By Property (i), we have $e_{\Gamma'}(X,Y) \ge \rho |X||Y|$.
Therefore since $\varepsilon N \le \frac{1}{2}\rho |Y|$,
\[
	e_{\Gamma_r}(X,Y) 
	\ge e_{\Gamma'}(X,Y) - |X| \cdot \varepsilon N
	\ge \rho |X||Y| - \frac{1}{2}\rho |X||Y| = \frac{1}{2}\rho |X||Y|.
\]
To prove the second part of the claim, fix two indices $i,j \in [s]$ satisfying $i < j$. 
By Property (ii), each vertex $u \in U_i$ has at least $(1-\beta)|U_j|$ neighbors in $U_j$ in $\Gamma'$.
Since $\varepsilon \le \alpha \beta$, we see that  $u$ has at least 
$(1-\beta)|U_j| - \varepsilon N \ge (1-2\beta)|U_j|$ neighbors in $U_j$, thus proving the claim.
\end{proof}

We also need the following theorem proved by Lov\'asz \cite{Lovasz}.

\begin{lem} \label{lem:Lovasz}
Let $G$ be a graph of maximum degree at most $\Delta$, and $d_1, \cdots, d_s$ be non-negative integers
satisfying $d_1 + \cdots + d_s \ge \Delta - t + 1$. Then there exists a vertex
partition $V(G) = V_1 \cup \cdots \cup V_s$ such that for all $i \in [s]$, the induced subgraph
$G[V_i]$ has maximum degree at most $d_i$.
\end{lem}

The following lemma is an embedding lemma for $(\alpha, \beta, \rho, \delta, \Delta)$-dense graphs.
It is a variant of \cite[Lemma 2.5]{CoFoSu} for weighted graphs.

\begin{lem} \label{lem:wr_embedding}
Let $\alpha, \rho, \delta$ be fixed positive real numbers satisfying $\rho \le \frac{1}{16}$
and $\delta \ge \frac{\rho}{2}$.
If $\Gamma$ is an $(\alpha, \frac{1}{2\Delta}, \rho, \delta, \Delta)$-dense graph on 
$N \ge 8\alpha^{-1}(\frac{2}{\delta})^{\Delta}n$ vertices, 
then $\Gamma$ contains a copy of every weighted graph $(G,w)$ of total weight at most $n$
and maximum degree at most $\Delta$.
\end{lem}
\begin{proof}
Note that if $\Delta=0$, then the conclusion trivially holds, and
hence we may assume that $\Delta \ge 1$.
By definition, there exists a sequence $U_1, U_2, \cdots, U_s$ of disjoint vertex subsets of $\Gamma$
each of cardinality at least $\alpha N$ and non-negative integers $\Delta_1, \ldots, \Delta_s$
such that $\Delta_1 + \cdots + \Delta_s = \Delta - s + 1$ for which
\begin{itemize}
  \setlength{\itemsep}{1pt} \setlength{\parskip}{0pt}
  \setlength{\parsep}{0pt}
\item[(i)] the induced subgraph $\Gamma[U_j]$ is bi-$(\rho^{2\Delta_j}, \delta)$-dense for each $j \in [s]$, and
\item[(ii)] for $1 \le j < j' \le s$, each vertex in $U_{j}$ has at least $\left(1-\frac{1}{2\Delta}\right)|U_{j'}|$ neighbors in $U_{j'}$.
\end{itemize}

Let $(G,w)$ be a weighted graph of total weight at most $n$
and maximum degree at most $\Delta$.
By Lemma~\ref{lem:Lovasz}, there exists a vertex partition $V(G) = V_1 \cup \cdots \cup V_s$
such that for all $j \in [s]$, the induced subgraph $G[V_j]$ has maximum degree at most $\Delta_j$.
Let $v_1, v_2, \cdots, v_n$ be an enumeration of the vertices of $G$ with the following properties:
\begin{itemize}
  \setlength{\itemsep}{1pt} \setlength{\parskip}{0pt}
  \setlength{\parsep}{0pt}
\item[(a)] For all $1 \le j < j' \le s$, the vertices in $V_j$ come before vertices in $V_{j'}$, and
\item[(b)] for all $j \in [s]$, the vertices in $V_j$ are ordered so that their weights form a non-increasing sequence.
\end{itemize}
For each $t \in [n]$, define $\pi(t) \in [s]$ as the index for which $v_t \in V_{\pi(t)}$.

Consider the following greedy algorithm of embedding the vertices of $G$, where the $t$-th step
of the algorithm selects the image of $v_t$ in $V(\Gamma)$. 
At time $t$, the algorithm is given as input a partial embedding $f$ 
defined on $\{v_1, \cdots, v_{t-1}\}$,
where at the initial step, $f$ is partial embedding of the empty graph.
For $t \in [n]$ and $i \ge t$, define $N_i^{(t)} = N(v_i) \cap \{v_1, \cdots, v_{t-1}\}$ 
as the set of neighbors of $v_i$ that precede $v_t$. 
Define $W_i^{(t)} = U_{\pi(i)} \cap \bigcap_{v \in N_i^{(t)}} N(f(v))$ and note that 
each vertex in $W_t^{(t)}$ can be used as the image of $v_t$ to extend the partial embedding.
For each $i \ge t$, define $d_i^{(t)} = |N_i^{(t)} \cap V_{\pi(i)}|$ as the number of neighbors of $v_i$
preceding itself in its own part.
Throughout the process, we will maintain the following property:
\begin{align} \label{eq:sizebound}
	\forall i \ge t, \quad |W_i^{(t)}| \ge \frac{1}{2}\left(\frac{\delta}{2}\right)^{d_i^{(t)}} |U_{\pi(i)}|.
\end{align}
Note that if $d_i^{(t)} = 0$, then \eqref{eq:sizebound} follows from Property (ii) since
$|N^{(t)}_i| \le \deg(v_t) \le \Delta$ holds and
\begin{align} \label{eq:sizebound_free}
	|W_i^{(t)}| \ge |U_{\pi(i)}| - \frac{1}{2\Delta}|U_{\pi(i)}| \cdot |N_i^{(t)}| \ge \frac{1}{2} |U_{\pi(i)}|.
\end{align}
Initially at $t=1$, we define $W_i^{(1)} = U_{\pi(i)}$ for all $i \in [n]$. 
Moreover since $d_i^{(1)} = 0$ for all $i \in [n]$, equation \eqref{eq:sizebound} holds.

Suppose that we are at the $t$-th step of the algorithm for some $t \in [n]$.
Define $I^+ = \{i > t \,:\, v_i \in N(v_t) \cap V_{\pi(t)} \}$.
Let $W \subseteq W_t^{(t)}$ be the set of vertices
$u \in W_t^{(t)}$  such that for all $i \in I^+$, 
$|N(u) \cap W_i^{(t)}| \ge \frac{\delta}{2} |W_i^{(t)}|$. 
If there exists a vertex $u \in W$ such that $w(f^{-1}(u)) \le 1 - w(v_t)$,
then define $f(v_t) = u$. This is a partial embedding since $f(v_t) \in W_t^{(t)}$
and $w(f^{-1}(u)) \le 1$.
Furthermore \eqref{eq:sizebound} is satisfied for $i > t$
having $\pi(i) > \pi(t)$ by \eqref{eq:sizebound_free}, and having $\pi(i) = \pi(t)$ 
but $i \notin I^+$ since $W_i^{(t+1)} = W_i^{(t)}$.
If $\pi(i) = \pi(t)$ and $i \in I^+$, then $d_i^{(t+1)} = d_i^{(t)} + 1$ and therefore \eqref{eq:sizebound}
holds since $W_i^{(t+1)} = N(u) \cap W_i^{(t)}$.

Therefore it suffices to prove the existence of a vertex $u \in W$ satisfying $w(f^{-1}(u)) \le 1 - w(v_t)$.
Suppose that all vertices $u \in W$ satisfy $w(f^{-1}(u)) > 1 - w(v_t)$.
Recall that $w(v_j) \ge w(v_t)$ for all $j \le t$
satisfying $\pi(j) = \pi(t)$ by Property (b).
Since $w(f^{-1}(u)) > 1 - w(v_t) \ge 0$ implies that $f^{-1}(u) \neq \emptyset$,
if $w(v_t) > \frac{1}{2}$, then it follows that $w(f^{-1}(u)) \ge w(v_t) > \frac{1}{2}$. 
On the other hand, if $w(v_t) \le \frac{1}{2}$, then $w(f^{-1}(u)) > 1 - w(v_t) \ge \frac{1}{2}$. 
Therefore for all vertices $u \in W$, we have $w(f^{-1}(u)) > \frac{1}{2}$. 
Since 
\[
	\frac{1}{2}|W| < \sum_{u \in W} w(f^{-1}(u)) \le w(V(G)) \le n,
\]
we see that $|W| < 2n$.

For all vertices $u \in W_t^{(t)} \setminus W$, there
exists $i \in I^+$ such that $|N(u) \cap W_i^{(t)}| < \frac{\delta}{2} |W_i^{(t)}|$. 
For notational simplicity, define $k = \Delta_{\pi(t)}$.
Since $|I_+| \le k$, by the pigeonhole principle, there exists an index $i_0 \in I^+$ such that
$|N(u) \cap W_{i_0}^{(t)}| < \frac{\delta}{2} |W_{i_0}^{(t)}|$ holds for at least
$\frac{1}{k}|W_t^{(t)} \setminus W|$ vertices. Let $X_1$ be the set of 
these vertices and note that
\[
	|X_1|
	\ge \frac{1}{k}(|W_t^{(t)}| - 2n)
	\ge \frac{1}{k}\left(\frac{1}{2}\left(\frac{\delta}{2}\right)^{k} |U_{\pi(t)}| - 2n\right)
	\ge \frac{1}{4k}\left(\frac{\delta}{2}\right)^{k} |U_{\pi(t)}|
	\ge \rho^{2k}|U_{\pi(t)}|,
\]
where the second to last inequality follows since $|U_{\pi(t)}| \ge \alpha N$
and the last inequality follows since $\rho \le \frac{1}{16}$ and $\rho \le 2\delta$. 

Define $X_2 = W_{i_0}^{(t)}$ and note that \eqref{eq:sizebound} implies 
$|X_2| \ge \frac{1}{2}(\frac{\delta}{2})^{k}|U_{\pi(t)}| \ge 2\rho^{2k} |U_{\pi(t)}|$. 
Let $X_1' \subseteq X_1$ be an arbitrary subset of size exactly $\rho^{2k}|U_{\pi(t)}|$, and
define $X_2' = X_2 \setminus X_1'$.  Then $|X_2'| \ge \frac{1}{2}|X_2| \ge \rho^{2k}|U_{\pi(t)}|$. 
Furthermore, each vertex $w \in X_1'$ has at most $\frac{\delta}{2}|X_2| \le \delta |X_2'|$ neighbors in $X_2'$.
This contradicts the fact that $\Gamma[U_{\pi(t)}]$ is bi-$(\rho^{2k}, \delta)$-dense.
Therefore there exists a vertex $u \in W$ satisfying  $w(f^{-1}(u)) \le 1 - w(v_t)$.
\end{proof}

Theorem~\ref{thm:weight_cfs} straightforwardly follows from Lemmas~\ref{lem:cfs_stable} and \ref{lem:wr_embedding}.

\begin{thm*}
There exists a constant $c > 1$ such that the following holds for all $\Delta$ and $\varepsilon$
satisfying $\varepsilon < c^{-\Delta \log \Delta}$.
If $(G, w)$ is a weighted graph with maximum degree at most $\Delta$ and total weight at most $n$,
then $\hat{r}_\varepsilon(G) \le c^{\Delta \log \Delta} n$.
\end{thm*}
\begin{proof}
Let $N = c^{\Delta \log \Delta} n$ for a constant $c$ to be chosen later.
Let $(G,w)$ be a weighted graph given as above. 
Suppose that $\Gamma$ is a graph on $N$ vertices with minimum degree at least $(1-\varepsilon)N$,
and consider an edge-coloring with two colors red and blue.
Let $h$ be the integer satisfying
$2^{h-1} \le \Delta \le 2^{h}-1$ and note that $h \le \log(2\Delta)$. Define $D = 2^{h}-1 \le 2\Delta$
and $\rho = \frac{1}{2^{h+4}\Delta} \ge \frac{1}{32\Delta^2}$.
If $c$ is sufficiently large, then $\varepsilon < 2^{-2h-1}\rho^{8D-4h+1}$ and thus
by Lemma~\ref{lem:cfs_stable}, we see that
the red graph or the blue graph is $(2^{-2h}\rho^{6D-4h}, 4(D+1)\rho, \rho, \frac{1}{2}\rho, D)$-dense.
Without loss of generality, assume that it is the red graph.
Then by monotonicity, the red graph is $(\rho^{12D}, \frac{1}{2\Delta}, \rho, \frac{1}{2}\rho, \Delta)$-dense.
Therefore if $c$ is large enough, then $N \ge 8 \rho^{-12D}(\frac{4}{\rho})^{\Delta}n$ and
by Lemma~\ref{lem:wr_embedding}, the red graph contains a copy of $G$.
\end{proof}

\section{A variant of the blow-up lemma}
\label{sec:blowup}

In this section, we prove Lemma~\ref{lem:blowup}, a variant of the blow-up lemma.
We will use a simplified version of the Random Greedy Algorithm (RGA)
developed by Koml\'os, S\'ark\"ozy, and Szemer\'edi \cite{KoSaSz}.
Their original algorithm consisted of two phases. In Phase 1, they embed the 
vertices one at a time, where at each step one considers all possible images 
that is consistent with the previous embedding and choose a random vertex among them.
Phase 1 continues until almost all vertices of the graph has been embedded. 
In Phase 2, they finish the embedding by invoking Hall's theorem.
For our proof, we do not need the second phase, 
since we only need an almost spanning embedding.
One can prove Lemma~\ref{lem:blowup} by carefully making this adjustent in their proof. 
It is rather straightforward to incorporate this change, but we include the proof here for completeness.

\medskip

Let $G$ be a graph with maximum degree at most $\Delta$.
Let $\Gamma$ be a graph with a vertex partition $\{V_i\}_{i=1}^{k}$ satisfying $|V_i| \ge (1+\xi)m$ for all $i \in [k]$, and let $R$ be its $(\varepsilon, \delta)$-reduced graph.
Suppose that there exists a homomorphism $f$ from $G$ to $R$ where for all $i \in [k]$, $|f^{-1}(i)| \le m$ for all $i \in [k]$.
For simplicity we assume that $|f^{-1}(i)| = m$ for all $i \in [k]$ by adding isolated vertices if necessary.
In order to avoid confusion, we will refer to the vertices in $G$ using $x,y$
and the vertices in $\Gamma$ using $v,w$.

Let $\varepsilon, \varepsilon_1, \varepsilon_2$ be positive real numbers 
satisfying $\varepsilon \ll \varepsilon_2 \ll \varepsilon_1$ where
$\varepsilon_1$ is small enough depending on $\delta$ and $\xi$. 
We first embed $f^{-1}(1)$ to $V_1$, then $f^{-1}(2)$ to $V_2$, and continue until 
we embed $f^{-1}(k)$ to $V_k$.
Suppose that we finished embedding $f^{-1}(i-1)$ to $V_{i-1}$ for some $i \in [k]$. 
Define $A_0 = f^{-1}(1) \cup \cdots \cup f^{-1}(i-1)$ and $B_0 =V(G) \setminus A_0$. 
For each $y \in B_0$, define $N_0(y) = N(y) \cap A_0$ as the set of 
neighbors of $y$ already embedded, and let $d_0(Y) = |N_0(y)|$.
Define $U_0(y) = V_{f(y)} \cap \bigcap_{z \in N_0(y)} N(\phi(z))$.
Consider the following property:
\begin{quote}
$\mathcal{P}(i)$ : For all $X \subseteq V_i$ of size $\varepsilon_1 |V_i| \le |X| \le m$, 
there are less than $\varepsilon_1 m$ vertices $y \in f^{-1}(i)$ such that $|U_0(y) \cap X| \ge (1- \varepsilon_1) |U_0(y)|$.
\end{quote}
We will show that there exists a random embedding algorithm that embeds $f^{-1}(i)$ to $V_i$ 
so that the probability that $\mathcal{P}(1), \cdots \mathcal{P}(i)$ hold 
but $\mathcal{P}(i+1)$ does not is small. 

Fix an arbitrary enumeration of the vertices in $f^{-1}(i)$.
We will iteratively embed the vertices of $f^{-1}(i)$ mainly following the order of this enumeration.
For some $s \ge 0$, suppose that we finished embedding $s$ vertices of $f^{-1}(i)$ and let 
$A_s \subseteq V(G)$ be the set of embedded vertices and $B_s = V(G) \setminus A_s$ be 
its complement. Hence $|A_s \setminus A_0| = s$. Let $\phi$ be the 
partial embedding of $G$ to $\Gamma$ defined on $A_s$.
We will maintain a first-in first-out queue $Q$ throughout the process, where initially $Q = \emptyset$. 
At the next step, if $Q \neq \emptyset$, then we let $x_{s}$ be the first vertex in $Q$, 
and if $Q = \emptyset$, then we let $x_{s}$ be the first non-embedded vertex according to the
enumeration given above. 
We will define the image of $x_{s}$ in the next step.

For each vertex $y \in B_s$, define $N_s(y) = N(y) \cap A_s$ as the set of 
neighbors of $y$ already embedded, and let $d_s(Y) = |N_s(y)|$.
Define $U_s(y) =  V_{f(y)} \cap \bigcap_{z \in N_s(y)} N(\phi(z))$.
Throughout the process, we will maintain the following properties:
\begin{itemize}
  \setlength{\itemsep}{1pt} \setlength{\parskip}{0pt}
  \setlength{\parsep}{0pt}
  \item[(i)] for all $y \in B_s$, we have $|U_s(y)| \ge (\delta - \varepsilon)^{d_s(y)}|V_{f(y)}|$,
  \item[(ii)] for all $y \in f^{-1}(i) \setminus Q$, we have $|U_s(y) \setminus \phi(A_s)| \ge \varepsilon_2 |V_{i}|$, and
  \item[(iii)] $|Q| \le \varepsilon_1 m$.
\end{itemize}
We will add a vertex to $Q$ when and only when (ii) fails. Thus Property (ii) 
always holds. We will later show that Property (i) is maintained by how we choose
the embedding $\phi$.
Note that since we are embedding
vertices in $f^{-1}(i)$ to $V_i$, the definition of $Q$ implies $Q \subseteq f^{-1}(i)$. 
Furthermore since $f^{-1}(i)$ is an independent set, 
for all $y \in B_s \cap f^{-1}(i)$, we have $N_s(y) = N_0(y)$ and hence $U_s(y) = U_0(y)$.
Thus $|U_t(y) \setminus \phi(A_t)|$ is non-increasing in time $t$. 
For $y \in Q$, since $|U_t(y) \setminus \phi(A_s)| < \varepsilon_2|V_{i}|$ at the time $t$ that $y$ was added to $Q$, it follows that if $y \in Q$ at time $s$, then 
$|U_s(y) \setminus \phi(A_s)| < \varepsilon_2 |V_{i}|$. 
For all $y \in Q$, since $N_s(y) = N_0(y)$ and $d_s(y) = d_0(y)$ for all $y \in Q$, by Property (i),
\begin{align*}
	|U_s(y) \cap \phi(A_s)| 
	=&\, |U_s(y)| - |U_s(y) \setminus \phi(A_s)| \\
	=&\, \left(1 - \frac{|U_s(y) \setminus \phi(A_s)|}{|U_s(y)|} \right) |U_0(y)| \\
	\ge&\, \left(1 - \frac{\varepsilon_2}{(\delta - \varepsilon)^{\Delta}} \right) |U_0(y)|
	\ge \left(1 - \varepsilon_1\right) |U_0(y)|.
\end{align*}
If $|V_i \cap \phi(A_s)| = s \ge \varepsilon_1 |V_i|$, then by Property $\mathcal{P}(i)$
it follows that $|Q| \le \varepsilon_1 m$. On the other hand if $s < \varepsilon_1 |V_i|$, then
for all $y \in B_s$, by Property (i) we have $|U_s(y) \setminus \phi(A_s)| \ge (\delta - \varepsilon)^{d_s(y)}|V_{f(y)}| - \varepsilon_1|V_{f(y)}| > \varepsilon_1|V_{f(y)}|$ and therefore $y \notin Q$. 
This implies that $Q = \emptyset$ if $s < \varepsilon_1|V_i|$. 
Therefore Property (iii) holds if $\mathcal{P}(i)$ holds.

Define $U = U_s(x_s)$. 
Note that if $x_s \notin Q$, then $|U \setminus A_s| \ge \varepsilon_2 |V_i|$. 
On the other hand, suppose that $x_s \in Q$. 
As observed above, $d_t(x_s)$ is constant for $t=0,1,2\cdots, m-1$. 
Therefore the size of $U \setminus A_s$ can change by at most one at each step. 
Since $Q$ is a first-in first-out queue, by Property (iii), 
there are at most $\varepsilon_1 m$ steps between the 
time that $x_s$ was first added to the queue and time $s$. This implies that
$|U \setminus A_s| \ge \varepsilon_2 |V_i| - \varepsilon_1 m \ge \frac{1}{2} \varepsilon_2 |V_i|$. 
Therefore in both cases, we have $|U \setminus A_s| \ge \frac{1}{2}\varepsilon_2 |V_i|$. 

For each $y \in N(x_s) \cap B_s$, since $f$ is a homomorphism from $G$ to $R$, 
we know that the pair $(V_i, V_{f(y)})$ is $\varepsilon$-regular of density at least $\delta$.
Moreover since $U \subseteq V_i, U_s(y) \subseteq V_{f(y)}$, and
$|U_s(y)| \ge (\delta - \varepsilon)^{d_s(y)}|V_{f(y)}|$ 
(by Property (i)), the set of vertices $Z_y \subseteq U$ with
less than $(\delta - \varepsilon)|U_s(y)|$ neighbors in $U_s(y)$ has size
$|Z_y| \le \varepsilon |V_{i}|$.
Define $U' = (U \setminus A_s) \setminus \bigcup_{y \in N(x_s) \cap B_s} Z_y$ and note that
\[
	|U'| \ge \frac{1}{2}\varepsilon_2 |V_i| - \Delta\varepsilon |V_i| \ge \frac{1}{4}\varepsilon_2 |V_i|. 
\]
Let $\phi(x_s)$ be a vertex in $U'$ chosen uniformly at random. 
The following lemma shows that Property $\mathcal{P}(i+1)$ holds with high probability 
after we finish embedding $f^{-1}(i)$ to $V_i$. 

\begin{lem} \label{lem:prob}
The probability that $\mathcal{P}(i+1)$ does not holds but $\mathcal{P}(1), \cdots, \mathcal{P}(i)$
holds is at most $e^{-\Omega(m)}$. 
\end{lem}

Given this lemma, by taking the union bound, we see that the probability that $\mathcal{P}(i)$ does not hold
for some $i$ is at most $ke^{-\Omega(m)} = o(1)$. Hence with non-zero probability, 
the algorithm will successfully terminate and embed $G$ to $\Gamma$.

\begin{proof}[Proof of Lemma~\ref{lem:prob}]
Let $\mathcal{E}$ be the event that $\mathcal{P}(1), \cdots, \mathcal{P}(i)$ holds.
Fix a set $X \subseteq V_{i+1}$ of size $\varepsilon_1 |V_i| \le |X| \le m$.
Define $A = \bigcup_{j=1}^{i} f^{-1}(j)$ and $B = V(G) \setminus A$. 
For each $y \in f^{-1}(i+1)$, define $U(y) = V_{i+1} \cap \bigcap_{z \in N(y) \cap A} N(\phi(z))$.
Let $R \subseteq f^{-1}(i+1)$ be a fixed set of size at least $\varepsilon_1 m$. 
We first compute the probability that 
\begin{quote}
(*) all vertices $y \in R$ satisfies $|U(y) \cap X| \ge (1-\varepsilon_1)|U_0(y)|$.
\end{quote}
Note that $\mathcal{P}(i+1)$ holds if there are no such pair of sets $(X,R)$.

Since $G$ has maximum degree at most $\Delta$, we can find a subset
$R' \subseteq R$ of size at least $\frac{|R|}{\Delta^2+1}$ 
whose pairwise distance is at least 3 in $G$. In other words, the sets $N(y)$ are disjoint
for vertices $y \in R'$. 
Fix a vertex $y \in R'$. We examine the probability that 
$|U(y) \cap X| \ge (1-\varepsilon_1)|U(y)|$. 
Let $z_1, z_2, \cdots, z_d$ be the vertices in $A \cap N(y)$ in the order of
embedding (note that $d \le \Delta$). Then
\[
	U(y) = V_{i+1} \cap N(\phi(z_1)) \cap N(\phi(z_2)) \cdots \cap N(\phi(z_d)).
\]
For $j=0,1,2,\cdots, d$, define $W_j(y) = V_{i+1} \cap N(\phi(z_i)) \cap \cdots \cap N(\phi(z_j))$. 
By the definition of our embedding algorithm, either $\mathcal{E}$ does not hold, 
or we have $|W_j(y)| \ge (\delta - \varepsilon)^j |V_{i+1}|$ for all $j=1,2,\cdots, d$. 
Since $U(y) = W_d(y)$, we have 
\[
	|U(y) \cap X| 
	\ge (1-\varepsilon_1) |U(y)|
	\ge (1-\varepsilon_1) \cdot (\delta - \varepsilon)^d |V_{i+1}|
	> (\delta + \varepsilon)^{d} |X|.
\]
Therefore there exists some $t$ such that 
\begin{align} \label{eq:breakpoint}
|W_{t}(y) \cap X| > (\delta+\varepsilon)^{t} |X|
\quad \textrm{ but }  \quad
|W_{t-1}(y) \cap X| \le (\delta+\varepsilon)^{t-1} |X|.
\end{align}
Since $|X| \ge \varepsilon_1 |V_{i+1}|$ and $(\delta+\varepsilon)^{\Delta}\varepsilon_1 \ge \varepsilon$, 
the above can hold only if $|W_{t}(y) \cap X| \ge \varepsilon |V_{i+1}|$.
This implies that $|W_{t-1}(y) \cap X| \ge \varepsilon |V_{i+1}|$.
Furthermore, since $f$ is a homomorphism from $G$ to $R$, the pair $(V_{i+1}, V_{f(z_t)})$ is
$\varepsilon$-regular with density at least $\delta$. Since $X \subseteq V_{i+1}$,
there are at most $\varepsilon|V_{f(z_t)}|$ vertices $z \in V_{f(z_t)}$ for which 
defining $\phi(z_t) = z$ would cause \eqref{eq:breakpoint}. 
Thus we can conclude that $y \in R'$ only if there exists $z_y \in A \cap N(y)$ whose image
$\phi(z_y)$ was chosen in a set of size at most $\varepsilon |V_{f(z_y)}|$. 

Therefore (*) holds only if for each $y \in R'$, there exists $z_y \in A \cap N(y)$ as above. 
On the other hand if $\mathcal{E}$ holds, 
then $\phi(z_y)$ was chosen inside a subset of $V_{f(z_y)}$ of size at least $\frac{1}{4}\varepsilon_2 |V_{f(z_y)}|$. 
Since the vertices in $R'$ have pairwise distance at least 3, all these vertices are 
distinct. Moreover, the number of choices of these vertices $z_y$ is at most $\Delta^{|R'|}$
and thus the probability that $\mathcal{E}$ holds and (*) holds is at most
\[
	\Delta^{|R'|} \cdot \left( \frac{\varepsilon}{\varepsilon_2 / 4} \right)^{|R'|}
	\le \left( \frac{4\varepsilon \Delta}{\varepsilon_2} \right)^{|R|/(\Delta^2 + 1)}
	\le \left( \frac{4\varepsilon \Delta}{\varepsilon_2} \right)^{\varepsilon_1 m/(\Delta^2 + 1)}.
\]

The number of choices for $R$ is at most $2^m$. 
Since the size of $X$ satisfies $\varepsilon_1|V_i| \le |X| \le m$, we must have
$|V_i| \le \varepsilon_1^{-1}m$ or otherwise the lemma is vacuously true.
Therefore the number of choices for the set $X$ is at most $2^{\varepsilon_1^{-1}m}$. 
Hence if $\varepsilon$ is sufficiently small, then the lemma follows from the union bound.
\end{proof}

\section{Bipartite graphs of small bandwidth} \label{sec:bipartite_ramsey}

In this section we prove Theorem~\ref{cor:bandwidth}.
Our proof is based on a variant on the idea independently used by Conlon \cite{Conlon}, 
and by Fox and Sudakov \cite{FoSu0} based on dependent random choice.
This variant of depenent random choice has been recently used in \cite{Lee}
to establish some embedding results for degenerate graphs.
The following lemma is the main ingredient of the proof.

\begin{lem} \label{lem:maxdegree_drc}
Let $G$ be an $n$-vertex graph of minimum degree at least $\alpha n$ 
and let $X_0$ be a subset of vertices.
For every positive real number $\beta$, there exists a set $X \subseteq V(G)$ satisfying the following properties:
\begin{itemize}
  \setlength{\itemsep}{1pt} \setlength{\parskip}{0pt}
  \setlength{\parsep}{0pt}
  \item[(i)] $|X| \ge \frac{1}{2}\alpha^{2\Delta} |V(G)|$,
  \item[(ii)] $|X \cap X_0| \ge \frac{1}{2}\alpha^{2\Delta} |X_0|$, and
  \item[(iii)] the number of $\Delta$-tuples in $X^\Delta$ with less than $\beta n$ common neighbors is at most $(\frac{2\beta}{\alpha^{2\Delta}}|X|)^{\Delta}$.
\end{itemize}
\end{lem}
\begin{proof}
Define $V = V(G)$. 
Choose $\Delta$ vertices ${\bf v_1, \ldots, v_\Delta} \in V$ independently and uniformly at random,
and let ${\bf X} = \bigcap_{i=1}^{\Delta} N({\bf v_i})$. 
By linearity of expectation,
\begin{align}
	\BBE\left[ \big| X_0 \cap {\bf X} \big| \cdot \big| {\bf X} \big| \right] 
	&= \sum_{x \in X_0, \,y \in V} \BFP(x, y \in {\bf X})
	= \sum_{x \in X_0, \,y \in V} \left(\frac{\codeg(x,y)}{n}\right)^{\Delta} \nonumber \\
	&\ge |X_0| n \left(\frac{1}{|X_0| n^2} \sum_{x \in X_0} \sum_{y \in V} \codeg(x,y) \right)^{\Delta}, \label{eq:maxdegree_drc_product}
\end{align}
where the inequality follows from convexity.
For a fixed vertex $x \in X_0$, the sum $\sum_{y\in V} \codeg(x,y)$ counts the number of walks of 
length 2 in $V$ that starts at $x$. Since $G$ has minimum degree at least $\alpha n$,
for all $x \in X_0$, we have $\sum_{y \in V}\codeg(x,y) \ge (\alpha n)^2$.
Hence from \eqref{eq:maxdegree_drc_product},
\[
	\BBE\left[ \big| X_0 \cap {\bf X} \big| \cdot \big| {\bf X} \big| \right] 
	\ge |X_0| n \left(\frac{ |X_0| \cdot \alpha^2 n^2}{|X_0| n^2} \right)^{\Delta}
	\ge \alpha^{2\Delta} |X_0|n,
\]
and by convexity,
\[
	\BBE\left[ \big| X_0 \cap {\bf X} \big|^{\Delta} \cdot \big| {\bf X} \big|^{\Delta} \right] 
	\ge \alpha^{2\Delta^2} |X_0|^{\Delta} n^{\Delta}.
\]

Call a $\Delta$-tuple of vertices {\em bad} if it has less than $\beta n$ common neighbors.
For a set $A$, define $\xi(A)$ as the number of bad $\Delta$-tuples in $A^\Delta$.
The probability of a fixed bad $\Delta$-tuple $T$ being in ${\bf X}^\Delta$
is at most $(\frac{\codeg(T)}{n})^\Delta \le \beta^\Delta$.
Hence by linearity of expectation, 
$\BBE[\xi({\bf X})]	\le	\beta^{\Delta} \cdot n^\Delta$.
Since
\begin{align*}
	\BBE\left[
		\frac{\big| X_0 \cap {\bf X} \big|^{\Delta} \cdot \big| {\bf X} \big|^{\Delta}}{\BBE[\big| X_0 \cap {\bf X} \big|^{\Delta} \cdot \big| {\bf X} \big|^{\Delta}]} 
		- \frac{\xi({\bf X}) \cdot \big|X_0 \cap {\bf X}\big|^{\Delta}}{2\BBE[\xi({\bf X})\cdot  \big|X_0 \cap {\bf X}\big|^{\Delta}]} 
	\right] = \frac{1}{2},
\end{align*}
there exists a set $X$ for which 
\begin{align*}
	\frac{\big| X_0 \cap X \big|^{\Delta} \cdot \big| X \big|^{\Delta}}{\BBE[\big| X_0 \cap {\bf X} \big|^{\Delta} \cdot \big| {\bf X} \big|^{\Delta}]} 
	- \frac{\xi(X) \cdot \big|X_0 \cap X\big|^{\Delta}}{2\BBE[\xi({\bf X}) \cdot \big|X_0 \cap {\bf X}\big|^{\Delta}]}
	\ge \frac{1}{2}.
\end{align*}
In particular,
\[
	\big| X_0 \cap X \big|^{\Delta} \cdot \big| X \big|^{\Delta}
	\ge \frac{1}{2} \BBE[\big| X_0 \cap {\bf X} \big|^{\Delta} \cdot \big| {\bf X} \big|^{\Delta}]
	\ge \frac{1}{2} \alpha^{2\Delta^2} |X_0|^{\Delta} n^{\Delta},
\]
and since $|X_0 \cap X| \le |X_0|$ and $|X| \le n$, this implies that
$|X| \ge \frac{\alpha^{2\Delta}}{2^{1/\Delta}} n$ and
$|X_0 \cap X| \ge \frac{\alpha^{2\Delta}}{2^{1/\Delta}} |X_0|$ thus proving
Properties (i) and (ii). Furthermore, 
\begin{align*}
	\xi(X) 
	&\le \big| X \big|^{\Delta} \frac{2\BBE[\xi({\bf X}) \cdot \big|X_0 \cap {\bf X}\big|^{\Delta}]}{\BBE[\big| X_0 \cap {\bf X} \big|^{\Delta} \cdot \big| {\bf X} \big|^{\Delta}]} 
	\le \big| X \big|^{\Delta} 
	\frac{2\beta^{\Delta} n^\Delta |X_0|^{\Delta}}{\alpha^{2\Delta^2} |X_0|^{\Delta} n^{\Delta}}
	\le |X|^{\Delta} \left(\frac{2\beta}{\alpha^{2\Delta}}\right)^{\Delta},
\end{align*}
and thus Property (iii) holds.
\end{proof}

We now prove Theorem~\ref{thm:densityembedding} using Lemma~\ref{lem:maxdegree_drc}.

\begin{thm*}
Let $\delta$ and $\alpha$ be positive real numbers. 
Let $G$ be an $n$-vertex graph of minimum degree at least $(\delta + \alpha)n$. Then $G$ contains
all bipartite graphs $H$ on at most $\delta n$ vertices with maximum degree at most $\Delta$ and bandwidth at most $\frac{1}{256\Delta}\alpha^{6\Delta+1}n$.
\end{thm*}
\begin{proof}
Let $G$ and $H$ be graphs given as above. Define $m = |V(H)|$.
Since $|V(H)| \le |V(G)|$, we can always embed the isolated vertices in the end.
Thus we may assume for simplicity that $H$ has no isolated vertex.
Let $V = V(G)$ and let $A \cup B$ the bipartition of $H$. 
Define $\beta = \frac{1}{256\Delta}\alpha^{6\Delta+1}$ and 
label the vertices of $H$ using $[m]$ 
so that $|i - j| \le \beta n$ whenever the vertices with labels $i$ and $j$ are adjacent.

For $t \ge 0$, define $B_t := [2t\beta n] \cap B$ and define $A_t$ as the set of vertices $a \in A$
for which $N_H(a) \subseteq B_t$. Note that since $H$ has bandwidth at most $\beta n$,
we have
$(A_{t+1} \cup B_{t+1}) \setminus (A_t \cup B_t)  \subseteq ((2t-3)\beta n, (2t+1)\beta n]$. Therefore
\begin{align} \label{eq:bandwidth_space}
	|(A_{t+1} \cup B_{t+1}) \setminus (A_t \cup B_t)| < 4\beta n
\end{align}
for all $t \ge 0$. Note that $A_0 = B_0 = \emptyset$ since $H$ has no isolated vertex.

We embed $H$ into $G$ using an iterative algorithm. 
Define $\gamma = 16\beta \alpha^{-2\Delta}$.
As an initialization, apply Lemma \ref{lem:maxdegree_drc} to $G$ with $(X_0)_{\ref{lem:maxdegree_drc}} = V$,
$\beta_{\ref{lem:maxdegree_drc}} = 8\beta$, and $\alpha_{\ref{lem:maxdegree_drc}} = \alpha$ to obtain 
a set $X_0$ (which is the set $X$ that we obtain by applying the lemma) of size 
$|X_0| \ge \frac{1}{2}\alpha^{2\Delta}n$ where the number of $\Delta$-tuples in $X^{\Delta}$
with less than $8\beta n$ common neighbors is at most $(\gamma |X_0|)^{\Delta}$.
Define $\phi$ as the trivial 
partial embedding of $H$ to $G$ defined on $A_0 \cup B_0 = \emptyset$.

For $t \ge 0$, at the $t$-th step of the algorithm, we are given as input 
a set $X_{t}$ and a partial embedding $\phi$ of $H$ to $G$ defined on $A_{t} \cup B_{t}$. 
Define $V_{t} = V \setminus \phi(A_{t} \cup B_{t})$.
We say that a $\Delta$-tuple of vertices $T$ is {\em $V_{t}$-bad} if the number of common neighbors
of $T$ in $V_t$ is less than $8\beta n$. Otherwise, we say that $T$ is {\em $V_{t}$-good}.
The given input satisfies the following properties:
\begin{itemize}
  \setlength{\itemsep}{1pt} \setlength{\parskip}{0pt}
  \setlength{\parsep}{0pt}
  \item[(a)] $X_{t} \subseteq V_{t-1}$,
  \item[(b)] $|X_{t}| \ge \frac{1}{2}\alpha^{2\Delta+1}n$,
  \item[(c)] $\phi(B_{t} \setminus B_{t-1}) \subseteq X_{t}$, and
  \item[(d)] for all $a \in A_{t+1} \setminus A_{t}$, the set $\phi(N(a) \cap B_{t})$ is contained in at most
  $(\gamma|X_{t}|)^{\Delta - |N(a) \cap B_t|}$  $V_{t-1}$-bad $\Delta$-tuples in $X_{t}$.
\end{itemize}
Note that the above properties hold for $t=0$ since
$N(a) \cap B_0 = \emptyset$ for all vertices $a$ (where we define $B_{-1} = V_{-1} = \emptyset$).
For some $t \ge 0$, suppose that we are given a set $X_{t}$ and a map $\phi$ defined on $A_{t} \cup B_{t}$ that satisfies the above properties.
Define $G_t$ as the subgraph of $G$ induced on $V_t = V \setminus \phi(A_t \cup B_t)$.
Since $|A_t \cup B_t| \le |V(H)| \le \delta n$, the given minimum degree condition on $G$ implies that
$G_t$ has minimum degree at least $\alpha n \ge \alpha |V(G_t)|$. In particular, this 
implies that $G_t$ has at least $\alpha n$ vertices.

Apply Lemma \ref{lem:maxdegree_drc} to $G_t$ with $(X_0)_{\ref{lem:maxdegree_drc}} = X_t \setminus \phi(A_t \cup B_t)$,
$\beta_{\ref{lem:maxdegree_drc}} = 8\beta$, and $\alpha_{\ref{lem:maxdegree_drc}} = \alpha$ to obtain a set $X_{t+1}$ satisfying the following properties:
\begin{itemize}
  \setlength{\itemsep}{1pt} \setlength{\parskip}{0pt}
  \setlength{\parsep}{0pt}
  \item[(i)] $|X_{t+1}| \ge \frac{1}{2} \alpha^{2\Delta} |V(G_t)| \ge \frac{1}{2} \alpha^{2\Delta+1}n$,
  \item[(ii)] $|X_t \cap X_{t+1}| 
	\ge \frac{1}{2} \alpha^{2\Delta} | (X_t \setminus \phi(A_t \cup B_t)) |
	\ge \frac{1}{2} \alpha^{2\Delta} (| X_t | - 4\beta n)
	\ge \frac{1}{8} \alpha^{4\Delta+1}n$, and
  \item[(iii)] the number of $V_{t}$-bad $\Delta$-tuples in $X_{t+1}^{\Delta}$ is at most $(\gamma|X_{t+1}|)^{\Delta}$,
\end{itemize}
Note that Properties (a) and (b) immediately follow.

To extend $\phi$ to $A_{t+1} \cup B_{t+1}$, we first extend $\phi$ to $B_{t+1} \setminus B_t$.
We embed vertices in $B_{t+1} \setminus B_t$ one at a time according to the order given by the labelling. Let $b \in B_{t+1} \setminus B_t$ be the current vertex where we identify $b$ with the integer in $[m]$. 
Define $B[b] = B \cap [b]$ and for each vertex $a \in A$, define $d_b(a) = |N(a) \cap B[b]|$.
We maintain the following three properties while extending $\phi$:
\begin{itemize}
  \setlength{\itemsep}{1pt} \setlength{\parskip}{0pt}
  \setlength{\parsep}{0pt}
  \item[(c')] $\phi(B[b] \setminus B_{t}) \subseteq X_{t} \cap X_{t+1}$,
  \item[(d1)] for all $a \in A_{t+1} \setminus A_{t}$, the set $\phi(N(a) \cap B[b])$ is contained in at most
    $(\gamma|X_t|)^{\Delta - d_b(a)}$ $V_{t-1}$-bad $\Delta$-tuples of $X_t$, and
  \item[(d2)] for all $a \in A_{t+2} \setminus A_{t+1}$, the set $\phi(N(a) \cap B[b])$ is contained in at most
  $(\gamma|X_{t+1}|)^{\Delta - d_b(a)}$ $V_t$-bad $\Delta$-tuples of $X_{t+1}$.
\end{itemize}
Initially, we may assume that $b = 2t\beta n$ so that $B[b] = B_t$.
Then Property (c') holds vacuously, and Property (d1) holds by Property (d) of the previous iteration. 
Moreover, note that if $a \in A_{t+2} \setminus A_{t+1}$, then $a$ is adjacent to a vertex in $B_{t+2}$, 
thus to a vertex with label at least $2(t+1)\beta n + 1$. Hence by the definition of bandwidth, it cannot be adjacent to 
a vertex in $B_{t}$, implying that $N(a) \cap B_t = \emptyset$. This implies (d2) at the initial stage,
by Property (iii).

Let $b \in B_{t+1} \setminus B_t$ be the next vertex to embed. 
Let $a_1, a_2, \ldots, a_d$ be the neighbors of $b$ (for $d \le \Delta$). 
Note that by the definition of $A_t$, we have $a_i \notin A_t$ for all $i \in [d]$. 
On the other hand for each $i \in [d]$, since $b \in B_{t+1} \subseteq [(2t+2)\beta n]$
and $H$ had bandwidth at most $\beta n$, the vertex $a_i$ cannot be adjacent to a vertex in $((2t+4)\beta n, m]$. This implies that $a_i \in A_{t+2} \setminus A_t$. 
For each $i \in [d]$, define $N_i = N(a_i) \cap [b-1]$ and note that $\phi$ is already defined on $N_i$.
For each $i \in [d]$, since $a_i$ is adjacent to $b$ and $H$ has bandwidth at most $\beta n$, 
the vertex $a_i$ cannot be adjacent to a vertex in $[b-2\beta n-1] \cap B \subseteq B_{t-1}$, thus implying that 
$N_i \subseteq B_{t+1} \setminus B_{t-1}$. 

Fix an index $i \in [d]$. If $a_i \in A_{t+1} \setminus A_t$, then
Property (c') implies that $\phi(N_i) \subseteq X_{t}$, and Property (d1) implies that
$\phi(N_i)$ is contained in at most $(\gamma|X_t|)^{\Delta - |N_i|}$ $V_{t-1}$-bad $\Delta$-tuples of $X_t$. Hence there are less than $\gamma|X_t|$ vertices $x \in X_t$
for which the $(|N_i|+1)$-tuple $N_i \cup \{x\}$ is contained in more than
$(\gamma|X_t|)^{\Delta - |N_i|-1}$ $V_{t-1}$-bad
$\Delta$-tuples of $X_t$.
If $a_i \in A_{t+2} \setminus A_{t+1}$, then Property (c') implies that
$\phi(N_i) \subseteq X_{t+1}$. Hence similarly as above Property (d2)
implies that there are less than $\gamma|X_{t+1}|$ vertices $x \in X_{t+1}$
for which the $(|N_i|+1)$-tuple $N_i \cup \{x\}$ is contained in more than
$(\gamma|X_{t+1}|)^{\Delta - |N_i|-1}$ $V_t$-bad $\Delta$-tuples of $X_{t+1}$. Since
\[
	|X_t \cap X_{t+1}| 
	\ge \frac{1}{8} \alpha^{4\Delta+1}n
	\ge 2\beta n + \frac{1}{16} \alpha^{4\Delta+1} n
	\ge 2\beta n + d \gamma n,
\]
we have $|(X_t \cap X_{t+1}) \setminus \phi(B[b-1])| \ge |X_t \cap X_{t+1}| - (2\beta n-1) \ge d \gamma n + 1$
(by Property (c')).
Therefore we can choose $\phi(b)=x$ to maintain Properties (d1) and (d2)
by avoiding the vertices identified above for each $i=1,2,\cdots,d$.

Once we finish embedding $B_{t+1}$, we greedily embed the vertices $a \in A_{t+1}$ one at a time.
Note that $\phi(N(a) \cap B)$ is contained in less than $(\gamma |X_{t+1}|)^{\Delta - |N(a) \cap B|} < |X_{t+1}|^{\Delta - |N(a) \cap B|}$ $V_{t-1}$-bad $\Delta$-tuples. Since the number of $\Delta$-tuples
containing $\phi(N(a) \cap B)$ is $|X_{t+1}|^{\Delta - |N(a) \cap B|}$, this in particular implies
that there exists a $V_{t-1}$-good $\Delta$-tuple containing $\phi(N(a) \cap B)$. 
Since every $V_{t-1}$-good tuple has at least $8\beta n$ common neighbors in $V_{t-1}$, we thus see 
that $\phi(N(a) \cap B)$ has at least $8\beta n$ common neighbors in $V_{t-1}$. 
By \eqref{eq:bandwidth_space}, we see that $|V_t \setminus V_{t-1}| \le 4\beta n$ and thus
$\phi(N(a) \cap B)$ has at least $4\beta n$ common neighbors in $V_t$. 
Therefore again by \eqref{eq:bandwidth_space}, we will never run out of vertices while greedily embedding the vertices in $A_{t+1}$ to appropriate vertices in $V_t$.
Note that Property (c) for the next step is satisfied by Property (c'), and 
Property (d) for the next step is satisfied by Property (d2). 
\end{proof}

Theorem~\ref{thm:densityembedding} has the following interesting corollary
which shows that a transference-type result holds even if
$\beta$ is as large as $c^{\Delta}$ for some constant $c$ when the given graph
is bipartite and has small bandwidth.

\begin{cor} \label{cor:bandwidth}
For every positive real number $\varepsilon$, there exists a real number $c < 1$ such that the following holds. 
If $G$ is a $n$-vertex bipartite graph of maximum degree at most $\Delta$ and bandwidth at most $c^{\Delta}n$, then $r(G) \le (4+\varepsilon)n$.
\end{cor}
\begin{proof}
Define $c = \frac{1}{256\Delta}\left(\frac{4(4+\varepsilon)}{\varepsilon}\right)^{6\Delta+1}$.
Let $N = (4+\varepsilon)n$ and suppose that the edge set of $K_N$ has been two-colored using red and blue.
Without loss of generality, we may assume that the red graph has density at least $\frac{1}{2}$.  
Then we can find a subgraph of the red graph having minimum degree at least 
$(1 + \frac{\varepsilon}{4})n$. Apply Theorem~\ref{thm:densityembedding}
to this graph with $n_{\ref{thm:densityembedding}} = N$, 
$\delta_{\ref{thm:densityembedding}} = \frac{1}{4+\varepsilon}$, and
$\alpha_{\ref{thm:densityembedding}} = \frac{\varepsilon}{4(4+\varepsilon)}$
to find a monochromatic copy of $G$.  
\end{proof}

\section{Concluding Remarks}
\label{sec:remarks}

The main theorem of this paper (Theorem~\ref{thm:transference}) is a transference
principle for Ramsey numbers of bounded degree graphs.
It asserts that for all $\Delta, \xi$ and $\varepsilon$, there exists $\beta$ and $n_0$ such that the following holds
for all $n \ge n_0$: if $G$ is a $n$-vertex graph of maximum degree at most $\Delta$
having a homomorphism $f$ to $H$ such that $|f^{-1}(v)| \le \beta n$ for all $v \in V(H)$, 
then $r(G) \le (1+\xi)\hat{r}_{\varepsilon}(H,w) \cdot \beta n$.
Similar result can be proved for more than two colors and for off-diagonal 
Ramsey numbers using the same approach.
The bound on $\beta$ that we obtain is of tower-type which is unlikely to be best possible. For example, Corollary~\ref{cor:bandwidth} shows that we may take
$\beta \le c^{\Delta}$ for some special case.

It might be the case that the transference principle holds for classes of graphs
more general than bounded degree graphs.

\begin{ques} \label{ques:degen}
Can Theorem~\ref{thm:transference} be extended to degenerate graphs?
\end{ques}

The main difficulty in following the same strategy used in this paper
lies in developing a variant of the blow-up lemma that we used.
In fact there has been some recent work on extending the blow-up lemma to classes 
of graphs beyond bounded degree graphs.
For an integer $a$, a graph is called $a$-arrangeable if its vertices can be ordered 
as $x_1, \cdots, x_n$ such that $|N(N(x_i) \cap R_i) \cap L_i\}| \le a$ for all $i \in [n]$.
where $R_i = \{x_{i+1}, \cdots, x_n\}$ and $L_i = \{x_1, \cdots, x_i\}$. 
B\"ottcher, Taraz, and W\"urfl \cite{BoTaWu}
extended the blow-up lemma to arrangeable graphs (after adding a 
weak constraint on the maximum degree). Their result implies that a 
transference-type result holds if the target graph $H$ is a bounded degree graph.
There also has been some partial success towards extending the blow-up lemma
to degenerate graphs \cite{Lee} but only when the bandwidth is small and for
almost spanning subgraphs.
It is plausible that some of the ideas used in these papers will help
answering Question~\ref{ques:degen}.

\medskip

Recall that for a given weighted graph $(G,w)$, $\hat{r}_\varepsilon(G)$ is not
necessarily finite if $\varepsilon$ is large. 
In fact $\hat{r}_\varepsilon(G)$ is finite if and only if $\varepsilon < \frac{1}{r(\chi(G))-1}$
(where $r(k)$ is the Ramsey number of $K_k$).
Let $s = r(\chi(G)) - 1$. 
If $\varepsilon \ge \frac{1}{s}$, then one can consider a red/blue coloring of $K_s$ with 
no monochromatic copy of $K_{\chi(G)}$ and take a balanced blow-up of this coloring
to find an arbitrarily large $n$-vertex graph with minimum degree at least $(1 - \frac{1}{s})n$
having no monochromatic subgraph of chromatic number at least $\chi(G)$.
In particular, it does not contain a monochromatic copy of $G$.
On the other hand if $\varepsilon < \frac{1}{s}$, then one can show that by supersaturation,
for sufficiently large $n$ there exists $\Omega(n^{\chi(G)})$ monochromatic copies of 
$K_{\chi(G)}$ in every red/blue coloring of an $n$-vertex graph $\Gamma$ of minimum degree at least $(1-\varepsilon)n$.
Without loss of generality, assume that at least half of such copies of $K_{\chi(G)}$ are red.
Consider a $\chi(G)$-uniform hypergraph over the vertex set of $\Gamma$ where we place a hyperedge
over all red copies of $K_{\chi(G)}$ in the coloring above. 
By K\"ov\'ari-S\'os-Tur\'an theorem for hypergraphs, we can find a complete $\chi(G)$-partite
graph with $|V(G)|$ vertices in each part if $n$ is sufficiently large. 
This implies that we can find a monochromatic copy of $G$ in $\Gamma$.

\medskip

\noindent {\bf Acknowledgements}. I thank David Conlon, Jacob Fox, and Benny Sudakov 
for fruitful discussions.

\end{document}

%% file: fig-hom.tex
\begin{tikzpicture}  \draw (-4.000, 0.000) circle (0.400);  \draw [fill=black] (4.000, 0.000) circle (0.1);  \draw (-4.000, 2.000) circle (0.250);  \draw [line width=2.5 pt] (-4.000,0.400) -- (-4.000, 1.750);  \draw [line width=2.5 pt] (-5.098,1.600) -- (-4.235, 1.914);  \draw [fill=black] (4.000, 2.000) circle (0.1);  \draw (4.000,0.000) -- (4.000, 2.000);  \draw (2.714,1.532) -- (4.000, 2.000);  \draw (-2.714, 1.532) circle (0.380);  \draw [line width=2.5 pt] (-3.743,0.306) -- (-2.959, 1.241);  \draw [line width=2.5 pt] (-3.765,1.914) -- (-3.072, 1.662);  \draw [fill=black] (5.286, 1.532) circle (0.1);  \draw (4.000,0.000) -- (5.286, 1.532);  \draw (4.000,2.000) -- (5.286, 1.532);  \draw (-2.030, 0.347) circle (0.450);  \draw [line width=2.5 pt] (-3.606,0.069) -- (-2.474, 0.269);  \draw [line width=2.5 pt] (-2.524,1.203) -- (-2.255, 0.737);  \draw [fill=black] (5.970, 0.347) circle (0.1);  \draw (4.000,0.000) -- (5.970, 0.347);  \draw (5.286,1.532) -- (5.970, 0.347);  \draw (-2.268, -1.000) circle (0.250);  \draw [line width=2.5 pt] (-3.654,-0.200) -- (-2.484, -0.875);  \draw [line width=2.5 pt] (-2.109,-0.096) -- (-2.225, -0.754);  \draw [fill=black] (5.732, -1.000) circle (0.1);  \draw (4.000,0.000) -- (5.732, -1.000);  \draw (5.970,0.347) -- (5.732, -1.000);  \draw (-3.316, -1.879) circle (0.300);  \draw [line width=2.5 pt] (-3.863,-0.376) -- (-3.419, -1.597);  \draw [line width=2.5 pt] (-2.459,-1.161) -- (-3.086, -1.687);  \draw [fill=black] (4.684, -1.879) circle (0.1);  \draw (4.000,0.000) -- (4.684, -1.879);  \draw (5.732,-1.000) -- (4.684, -1.879);  \draw (-4.684, -1.879) circle (0.550);  \draw [line width=2.5 pt] (-4.137,-0.376) -- (-4.496, -1.363);  \draw [line width=2.5 pt] (-3.616,-1.879) -- (-4.134, -1.879);  \draw [fill=black] (3.316, -1.879) circle (0.1);  \draw (4.000,0.000) -- (3.316, -1.879);  \draw (4.684,-1.879) -- (3.316, -1.879);  \draw (-5.732, -1.000) circle (0.150);  \draw [line width=2.5 pt] (-4.346,-0.200) -- (-5.602, -0.925);  \draw [line width=2.5 pt] (-5.105,-1.526) -- (-5.617, -1.096);  \draw [fill=black] (2.268, -1.000) circle (0.1);  \draw (4.000,0.000) -- (2.268, -1.000);  \draw (3.316,-1.879) -- (2.268, -1.000);  \draw (-5.970, 0.347) circle (0.400);  \draw [line width=2.5 pt] (-4.394,0.069) -- (-5.576, 0.278);  \draw [line width=2.5 pt] (-5.758,-0.852) -- (-5.900, -0.047);  \draw [fill=black] (2.030, 0.347) circle (0.1);  \draw (4.000,0.000) -- (2.030, 0.347);  \draw (2.268,-1.000) -- (2.030, 0.347);  \draw (-5.286, 1.532) circle (0.200);  \draw [line width=2.5 pt] (-4.257,0.306) -- (-5.157, 1.379);  \draw [line width=2.5 pt] (-5.770,0.694) -- (-5.386, 1.359);  \draw [fill=black] (2.714, 1.532) circle (0.1);  \draw (4.000,0.000) -- (2.714, 1.532);  \draw (2.030,0.347) -- (2.714, 1.532);  \draw[->] [line width=1.5 pt] (-1.000, 0.000) -- (1.000, 0.000);  \draw (-2.000, -1.800) node {$G$};  \draw (6.000, -1.800) node {$H$};\end{tikzpicture}

%% file: RamseyTransference.bbl
\begin{thebibliography}{}

\bibitem{AlBrSk}
P.~Allen, G.~Brightwell, and J.~Skokan,
Ramsey-goodness -- and otherwise,
{\it Combinatorica} {\bf 33} (2013), 125--160.

\bibitem{AlSp}
N.~Alon and J.~Spencer,
The probabilistic method,
2nd ed., Wiley, New York (2000).

\bibitem{BoPrTaWu}
J.~B\"ottcher, K.~Pruessmann, A.~Taraz, and A.~W\"urfl,
Bandwidth, expansion, treewidth, separators, and universality for bounded degree graphs,
{\it European J. of Combin.} {\bf 31} (2010), 1217--1227.

\bibitem{BoScTa}
J.~B\"ottcher, M.~Schacht, and A.~Taraz, 
Proof of the bandwidth conjecture of Bollob\'as and Koml\'os,
{\it Mathematische Annalen} {\bf 343} (2009), 175--205.

\bibitem{BoTaWu}
J.~B\"ottcher, A.~Taraz, and A.~W\"urfl,
Spanning embeddings of arrangeable graphs with sublinear bandwidth,
arXiv:1305.2078 [math.CO].

\bibitem{BuEr}
S.~Burr and P.~Erd\H{o}s,
On the magnitude of generalized Ramsey numbers for graphs, 
in {\it Infinite and Finite Sets I}, Colloq. Math. Soc. Janos Bolyai 10, North-Holland, Amsterdam (1975), 214--240.

\bibitem{ChRoSzTr}
V.~Chv\'atal, V.~R\"odl, E.~Szemer\'edi, and W.~Trotter,
The Ramsey number of a graph with bounded maximum degree,
{\it J. Combin. Theory B}, {\bf 34} (1983), 239--243.

\bibitem{Conlon}
D.~Conlon,
Hypergraph packing and sparse bipartite Ramsey numbers, 
{\it Combin. Probab. Comput.} {\bf 18} (2009), 913--923.

\bibitem{CoFoSu}
D.~Conlon, J.~Fox, and B.~Sudakov,
On two problems in graph Ramsey theory,
{\it Combinatorica} {\bf 32} (2012), 513--535.

\bibitem{CoFoSu15}
D.~Conlon, J.~Fox, and B.~Sudakov,
Recent developments in graph Ramsey theory,
arXiv:1501.02474 [math.CO].

\bibitem{Eaton}
N.~Eaton,
Ramsey numbers for sparse graphs,
{\it Discrete Math.} {\bf 185}, 63--75.

\bibitem{Erdos}
P.~Erd\H{o}s,
Some remarks on the theory of graphs,
{\it Bull. Amer. Math. Soc.} {\bf 53} (1947), 292--294.

\bibitem{ErSz}
P.~Erd\H{o}s and G.~Szekeres,
A combinatorial problem in geometry,
{\it Compositio Mathematica} {\bf 2} (1935), 463--470.

\bibitem{GeKoRoSt}
S.~Gerke, Y.~Kohayakawa, V.~R\"odl, and A.~Steger,
Small subsets inherit sparse $\varepsilon$-regularity,
{\it J. Combin. Theory B}, {\bf 97} (2007), 34--56.

\bibitem{GrRoRu00}
R.~Graham, V.~R\"odl, and A.~Ruci\'nski,
On graphs with linear Ramsey numbers,
{\it J. Graph Theory} {\bf 35} (2000), 176--192.

\bibitem{GrRoRu01}
R.~Graham, V.~R\"odl, and A.~Ruci\'nski,
On bipartite graphs with linear Ramsey numbers,
{\it Combinatorica} {\bf 21} (2001), 199--209.

\bibitem{GrRoSp}
R.~Graham, B.~Rothschild, and J.~Spencer,
Ramsey theory,
Wiley, New York (1990).

\bibitem{FoSu0}
J.~Fox and B.~Sudakov,
Density theorems for bipartite graphs and related Ramsey-type results,
{\it Combinatorica} {\bf 29} (2009), 153--196.

\bibitem{FoSu}
J.~Fox and B.~Sudakov,
Two remarks on the Burr-Erd\H{o}s conjecture,
{\it European J. Combin.} {\bf 30} (2009), 1630--1645.

\bibitem{KoSaSz}
J.~Koml\'os, G.~S\'ark\"ozy, and E.~Szemer\'edi,
Blow-up Lemma,
{\it Combinatorica} {\bf 17} (1997), 109--123.

\bibitem{KoSi}
J.~Koml\'os and M.~Simonovits,
Szemer\'edi's regularity lemma and its applications in graph theory,
Bolyai Society Mathematical Studies 2, Combinatorics, Paul Erd\H{o}s is Eighty (Volume 2) (D. Mikl\'os, V. T. S\'os, T. Sz\"onyi eds.), Budapest (1996), 295--352.

\bibitem{Lee}
C.~Lee,
Embedding degenerate graphs of small bandwidth,
arXiv:1501.05350 [math.CO].

\bibitem{Lovasz}
L.~Lov\'asz, 
On decomposition of graphs,
{\it Studia Sci. Math. Hungar.} {\bf 1} (1966), 237--238.

\bibitem{Ramsey}
F.~Ramsey,
On a problem of formal logic,
{\it Proc. London Math. Soc.} {\bf 30} (1930), 264--286.


\end{thebibliography}
